\documentclass[12pt, a4paper]{mmaauth_m}
\usepackage{color}

 \newtheorem{thm}{Theorem}[section]
 \newtheorem{cor}[thm]{Corollary}
 \newtheorem{lem}[thm]{Lemma}

 \newtheorem{defn}[thm]{Definition}
 \newtheorem{rem}[thm]{Remark}
\numberwithin{equation}{section}
\newenvironment{proof}{ \par\addvspace{3mm}\noindent {\bf Proof.\ }}
{\hfill$\square$ \par \addvspace{3mm} }
\newcommand{\qed}{\hfill$\square$}
\begin{document}
\runninghead{O.~Chkadua, S.~E.~Mikhailov, D.~Natroshvili}

\title 
{Localized boundary-domain singular 
integral equations of  Dirichlet problem
for  self-adjoint second order 
strongly elliptic PDE systems
 }

\author{O. Chkadua\affil{a}, S.E. Mikhailov\affil{b}\corrauth\ and D. Natroshvili\affil{c}}

\address{\affilnum{a} A.Razmadze  Mathematical Institute of I.Javakhishvili Tbilisi State University, 2 University str., Tbilisi 0186, Georgia\\ 
\& Sokhumi State University, 9 A.Politkovskaia str.,Tbilisi 0186, Georgia.\\
\affilnum{b} 
Department of Mathematics, Brunel University London,
Uxbridge, UB8 3PH,  UK.\\
\affilnum{c} 
Department of Mathematics, Georgian Technical University, 77,
M.Kostava str.,  Tbilisi 0175, Georgia\\
 \& I.Vekua Institute of Applied Mathematics of I.Javakhishvili Tbilisi State University, 2 University str., Tbilisi 0186, Georgia.}

\corraddr{S.E.Mikhailov, Department of Mathematics, Brunel University London,
Uxbridge, UB8 3PH,  UK,\\
E-mail: {Sergey.Mikhailov@brunel.ac.uk\\
\rm This is an open access article under the terms of the Creative Commons Attribution License, which permits use, distribution and reproduction in any medium,
provided the original work is properly cited.}
}
\date{Month, Day, Year}

     \MOS{35J25, 31B10, 45K05, 45A05}

     \begin{abstract}
     The paper deals with  the three--dimensional Dirichlet  boundary value problem (BVP) for a second
order strongly elliptic self-adjoint system of partial differential equations
in the divergence form with  variable   coefficients and develops
the integral potential method based on a localized parametrix.
Using Green's representation formula and
properties of the localized layer and volume potentials, we reduce the Dirichlet
BVP  to a system of localized boundary-domain  integral equations (LBDIEs).
The equivalence between the Dirichlet BVP and the corresponding LBDIE system is studied.
We establish that the obtained localized
boundary-domain integral operator belongs to the Boutet de Monvel
algebra. With the help of the Wiener-Hopf   factorization method
 we investigate corresponding Fredholm properties and prove invertibility of the
 localized operator in appropriate Sobolev (Bessel potential) spaces.
     \end{abstract}
\maketitle
\section{Introduction}
We consider the  Dirichlet  boundary-value problem (BVP) for a second
order strongly elliptic self-adjoint system of partial differential equations in the divergence form with  variable   coefficients and develop
the generalized integral potential method based on a {\em localized parametrix}.

The BVP treated in the paper is well investigated
in the literature by the variational method and also by the
classical integral potential method, when  the corresponding fundamental
solution is available in explicit form (see, e.g., \cite{HW}, \cite{LiMa}, \cite{McL}) or when at least its properties are known to be good enough (see, e.g., \cite{MT1999}, \cite{KPW2010} and references therein).

Our goal here is to develop a localized integral potential method for general second order strongly elliptic self-adjoint systems of partial differential equations with variable coefficients. 
We show that a solution of the problem can be represented by {\it explicit localized parametrix-based potentials} and that the corresponding
{\it localized boundary-domain  integral operator}
(LBDIO) is invertible, which is important for analysis of convergence and stability of LBDIE-based numerical methods for PDEs
 (see e.g. \cite{GMR2013}, \cite{Mik1}, \cite{MikNakJEM}, \cite{Po1}, \cite{SSA2000}, \cite{Taigbenu1999}, \cite{ZZA1998}, \cite{ZZA1999}).

Using Green's representation formula and
properties of the localized layer and volume potentials we reduce the Dirichlet
BVP  to a system of Localized Boundary-Domain  Integral Equations (LBDIEs).
First, we establish the equivalence between the original boundary value  problem and the corresponding
LBDIE system, which appeared to be quite non-trivial task and plays a crucial role in our analysis.
Afterwards, we establish that the localized boundary domain  integral
operator of the system belongs to the Boutet de Monvel operator algebra. Employing the Vishik-Eskin theory, based on the Wiener-Hopf factorization  method,
 we investigate corresponding Fredholm properties and prove invertibility of the localized
operator in appropriate Sobolev (Bessel potential) spaces.

In the references \cite{CMN-1}--\cite{CMN-IEOT2013}, \cite{MikMMAS2006},  the LBDIE method has been developed  for the case of scalar elliptic second order  partial differential equations with variable coefficients, and here we extend it to PDE systems.

\section{Boundary value problem and parametrix-based operators}
\subsection{Formulation of the boundary value problems and localized Green's third identity}
\label{ss2.1}
Consider a uniformly strongly elliptic second order self-adjoint matrix partial differential operator
\begin{equation}
\label{L-1}
A=A(x,\partial_x)=\big[ A_{pq}(x,\partial_x) \big]^3_{p,q=1}
= \Big[\frac{\partial}{\partial x_k}\,
\Big(a_{kj}^{pq}(x)\,\frac{\partial }{\partial x_j}\Big)\Big]^3_{p,q=1},
\end{equation}
where  $\partial_x=(\partial_1,\partial_2,\partial_3)$, $\partial_j=\partial_{x_j}=\partial/\partial x_j$,
$a_{kj}^{pq}=a_{jk}^{qp}=a^{kq}_{pj}\in C^\infty$,
 $j,k,p,q=1,2,3$.
Here and in what follows, the Einstein summation by repeated indices from $1$ to $3$ is assumed
if not otherwise stated.

  We assume that the coefficients  $a_{kj}^{pq}$ are real and the quadratic from $a_{kj}^{pq}(x)\,\eta_{kp}\,\eta_{qj}$ is uniformly positive definite  with respect to symmetric variables  $\eta_{kp}=\eta_{pk} \in \mathbb{R}$, which implies that    the principal homogeneous symbol of the operator $A(x,\partial_x)$ with opposite sign, $A(x,\xi)=[a_{kj}^{pq}(x)\xi_k\,\xi_j]_{3\times3}$  is uniformly positive definite, which for the real symmetric coefficients $a_{kj}^{pq}$ means there are positive constants $c_1$ and $c_2$   such that
  \begin{gather}
\label{1-d}
c_1\,|\xi|^2 |\zeta|^2\leq \bar\zeta \cdot A(x,\xi)\zeta\,\leq c_2\,|\xi|^2 |\zeta|^2 \quad
 \forall\;  x\in \mathbb{R}^3, \;\; \forall\;\xi\in \mathbb{R}^3,\;\;\forall\; \zeta\in \mathbb{C}^3. 
\end{gather}
Here $a \cdot b:=a^\top b:=\sum_{j=1}^3 a_j {b_j}$  is
the bilinear product of two column-vectors $a,\,b\in \mathbb{C}^3$.

Further, let $\Omega=\Omega^+$  be a bounded domain
 in $\mathbb{R}^3$ with a simply connected boundary  $\partial\Omega=S \in C^\infty$,
$\overline{\Omega}=\Omega\cup S$.
  Throughout the paper $n=(n_1,n_2,n_3)$ denotes the  unit normal vector to $S$  directed
outward the domain $\Omega$. Set $\Omega^-:=\mathbb{R}^3\setminus\overline{\Omega}$.

By  $H^r(\Omega)=H^r_2(\Omega)$ and $H^r(S)=H^r_2(S)$, $r\in \mathbb{R}$,  we denote the
Bessel potential spaces on a domain $\Omega$ and on a closed manifold $S$ without boundary, while
$\mathcal{D}( \mathbb{R} ^3)$ and $\mathcal{D}(\Omega)$ stand for $C^\infty$ functions
with compact support in $\mathbb{R}^3$
and in $\Omega$ respectively, and
$\mathcal{S}(\mathbb{R}^3)$ denotes the Schwartz space of rapidly decreasing functions in $\mathbb{R}^3$.
Recall that $H^0(\Omega)=L_2(\Omega)$ is a space of square integrable functions in $\Omega$.
For a vector $u=(u_1,u_2,u_3)^\top$ the inclusion $u=(u_1,u_2,u_3)^\top \in H^r $ means that each component $u_j$ belongs to the space $H^r $.

Let us denote by $\gamma^+u$ and $\gamma^-u$ the traces of  $u$  on  $S$
from the interior and exterior of  $\Omega^{+}$ respectively.

We also need the following subspace of $H^1(\Omega)$, see e.g. \cite{Costabel1988},
\begin{equation}
\label{3-d}
H^{1,\,0}(\Omega; A ):=\{u\!=\!(u_1,u_2,u_3)^\top\in H^{1}(\Omega ):  Au\in H^{0}(\Omega )\}.
\end{equation}

The Dirichlet boundary-value problem reads as follows:\\
{\it
Find a vector-function
 $u=(u_1,u_2,u_3)^\top \in H^{1,\,0}(\Omega,A)$ satisfying the differential equation
\begin{eqnarray}
\label{1}
Au=f \;\;\text{in} \;\; \Omega
\end{eqnarray}
and the Dirichlet boundary condition
\begin{eqnarray}
\label{2}
\gamma^+u=\varphi_{_0}      \;\;\text{on} \;\; S,
\end{eqnarray}
where  $\varphi_{_0}=(\varphi_{_{01}},\varphi_{_{02}},\varphi_{_{03}})^\top\in H^{1/2}(S)$
and $f=(f_1,f_2,f_3)^\top\in H^0(\Omega)$ are given vector functions.
}
Equation \eqref{1} is  understood in the distributional sense, while the
Dirichlet  boundary condition \eqref{2} is understood in the usual trace sense.

The {\it classical co-normal derivative operators}, $T^\pm$,  associated with the differential operator $A(x, \partial_x)$, are well defined in terms of the gradient traces on the boundary $S$ for a sufficiently smooth vector-function $v$, say
$v\in H^2(\Omega)$, as follows
\begin{align}
\label{2.4}
  [\,T^{\pm}(x, \partial_x)\,v(x)\,]_p:=
a^{pq} _{kj}(x) \,n_k(x)\,\gamma^\pm\partial_{x_j}v_q(x), \;\; x\in S,\;\; p=1,2,3.
 \end{align}

The co-normal derivative operator defined in \eqref{2.4} can be extended by
continuity to the space $H^{1,\,0}(\Omega; A)$. The extension is inspired by Green's first
identity  (cf. \cite{Costabel1988}, \cite{McL}, \cite{MikJMAA2011}) as follows,
\begin{align}
\label{4-d}
\langle T ^+\,v \,,\,g\rangle_{S } :=
\int\limits_{\Omega }[\gamma^{-1}g(x)] \cdot A (x,\partial_x)v (x)\,dx
+\!
\int\limits_{\Omega } E(v(x),\gamma^{-1}g(x))\,dx,\
  \forall g\in H^{1/2}(S),\ \forall v\in H^{1,\,0}(\Omega; A),
  \end{align}
where $\langle\cdot\,,\,\cdot\rangle_{S}$ denotes the duality between the adjoint spaces
$H^{-\frac{1}{2}}(S)$ and $H^{\frac{1}{2}}(S)$, which extends the usual bilinear $L_2(S)$ inner product,  while $E(v(x),u(x))=a _{kj}^{pq}(x)\,[\partial_{ x_j}v_q (x)]\, [\partial_{ x_k}u_p (x)]$.
By $\gamma^{-1}$ we denote a (non-unique) continuous linear extension operator acting from
  $H^{\frac{1}{2}}(S)$ into  $H^1(\mathbb{R}^3)$.
The  restrictions of $\gamma^{-1}$
  on $\Omega^+$ and $\Omega^-$ are the  right inverse operators to the
  corresponding trace operators $\gamma^+$ and $\gamma^-$.
Clearly definition \eqref{4-d} does not depend on the extension operator.

Moreover, by \cite[Lemma 3.4]{Costabel1988}, \cite[Lemma 4.3]{McL}), for any $v\in H^{1,0}(\Omega;A)$ and $u\in  H^1(\Omega)$ the first Green
identity holds in the form
\begin{equation}
\label{1GI}
\left\langle  T^{+}v\,,\, \gamma^+ u \right\rangle _{_{S}}
=
\int_{\Omega}\big[\, { u\cdot Av} + E(v,u)\,\big]\,
dx.
\end{equation}

\begin{rem}
\label{rBVP}
From the condition \eqref{1-d} it follows that the quadratic form $E(u(x),u(x))$ rewritten as
$$ 
E(u(x),u(x))= a _{kj}^{pq}(x)\,\varepsilon_{qj}(x)\, \varepsilon_{pk}(x)
$$
\text{where} \;\;
$$\varepsilon_{qj}(x)= \big(\partial_ju_q (x) +
\partial_q u_j (x)\big)/2,
$$ 
is positive definite in the symmetric variables $\varepsilon_{qj}$.
Therefore  Green's first identity \eqref{1GI}  and Korn's inequality along with the Lax-Milgram lemma imply that the Dirichlet
 BVP \eqref{1}-\eqref{2} is uniquely solvable in the space $H^{1,\,0}(\Omega; A)$
 (see, e.g., {\cite{Vishik1951}}, \cite{HW}, \cite{LiMa}, \cite{McL}).
\end{rem}
\subsection {Parametrix-based operators and integral identities}
As it has already been mentioned, our goal here is to develop the LBDIE method for the Dirichlet
 BVP \eqref{1}-\eqref{2}.

Let $F_\Delta(x):=-1/[\,4\,\pi\,|x|\,]$ denote the scalar fundamental solution of the Laplace operator,
 $\Delta=\partial^2_1+\partial^2_2+\partial^2_3$. Let us define a \textit{localized matrix parametrix} for the the matrix operator $I\Delta$ as
\begin{align}\label{2.17}
    P(x) \equiv P_{\chi}(x):=  P_\Delta(x)\,I=\chi (x)\,F_\Delta(x)\,I 
     =-\frac{\chi (x)}{4\,\pi\,|x|}\,I
 \end{align}
where $P_\Delta(x)\equiv P_{\chi\Delta}(x) :=\chi (x)\,F_\Delta(x)$ is a scalar function of the vector argument $x$, $I$ is the unit $3\times 3$ matrix, and $\chi$ is a localizing function (see Appendix A)
\begin{eqnarray}
 \label{2.18}
    \chi \in X^k_{+}\,, \;\;\;\;k\geq 3, 
         \;\;\text{with}\;\;\chi (0)=1,
\end{eqnarray}
Throughout the paper we assume that condition \eqref{2.18} is satisfied if not otherwise stated. Note that the function
${\chi}$ can have a compact support, which is useful for numerical implementations, but generally this is not necessary and the class $X^k_{+}$ include also the functions not compactly supported but sufficiently fast decreasing at infinity, see  \cite{CMN-Loc2} and Appendix A below for details. 

For sufficiently smooth vector-functions $u$ and $v$, say $u,v\in C^2(\overline{\Omega})$,
there holds Green's second identity
\begin{align}
\label{2.19}
\int _{\Omega }\big[ v \cdot A (x,\partial_x)u    - u\cdot A (x,\partial_x)v \big]\,dx  
=
\int _{S} \big[\gamma^+v \cdot T^+ u   -T^+v \cdot \gamma^+ u   \big]\,dS. 
\end{align}

Denote by $B(y,\varepsilon)$ a ball centered at point $y$, with radius $\varepsilon>0$, and let $\Sigma(y,\varepsilon):=\partial B(y,\varepsilon)$.
Let us take as $v(x)$, successively, the columns of the matrix  $P(x-y)$,
where $y$  is an arbitrarily fixed interior point
in $\Omega$, and write the   identity \eqref{2.19} for the region
$\Omega_\varepsilon:=\Omega\setminus B(y,\varepsilon)$ with $\varepsilon>0$  such that $\overline{B(y,\varepsilon)}\subset \Omega$. Keeping in mind
that $P^\top(x-y)=P(x-y)$ and $[A (x,\partial_x)P(x-y)]^\top=[A (x,\partial_x)P(x-y)]$,  we arrive at the equality,
\allowdisplaybreaks
\begin{multline}
\label{2.20}
\int _{\Omega_\varepsilon }\big [P(x-y) \; A (x,\partial_x)u(x) -
\{A (x,\partial_x)P(x-y)\} \;u(x) \big]\,dx= \\
\shoveleft{\int _{S}\big [ P(x-y)\; T^+(x,\partial_x) u(x) -
 \{T(x,\partial_x)P(x-y)\}^\top\gamma^+ u(x)  \big]\,d S_x}  \\
-\int _{\Sigma(y,\varepsilon)}\big [ P(x-y)   \;T^+(x,\partial_x) u (x)
  - \{T(x,\partial_x)P(x-y)\}^\top \gamma^+u (x) \big]\,d S_x.
\end{multline}
The normal vector on $\Sigma(y,\varepsilon)$ is directed inward $\Omega_\varepsilon$.

Let the operator $\mathcal{N}$ defined as
 \begin{align}
\label{2.21}
 \mathcal{N}\,u(y): & ={\rm{v.p.}}\int _{\Omega  }\;[\,A (x,\partial_x)P(x-y)\,]\, u(x) \,dx 
:  =\lim\limits_{\varepsilon \to 0}\int _{\Omega_\varepsilon }\;[\,A (x,\partial_x)P(x-y)\,]\, u(x) \,dx
\end{align}
be the Cauchy principal value singular integral operator, which is well defined if the limit in the right hand side exists.
The similar operator with integration over the whole space ${\mathbb{R}}^3$ is denoted as
\begin{equation}
\label{2.21-1d}
\displaystyle
{\bf N}\,u(y):={\rm{v.p.}}\int _{\mathbb{R}^3}\;[\,A (x,\partial_x)P(x-y)\,]\, u(x) \,dx\,.
\end{equation}
Note that
\begin{equation}
\label{2.22}
\displaystyle
 \frac{\partial^2}{\partial x_k\,\partial x_j}\,\frac{1}{|x-y|}=
 - \frac{4\,\pi\, \delta_{kj}}{3}\;\delta(x-y)+ {\rm{v.p.}} \,\frac{\partial^2}{\partial x_k\,\partial x_j}\frac{1}{|x-y|},
\end{equation}
where $\delta_{kj}$ is the Kronecker delta, and $\delta(\,\cdot\,)$ is the Dirac distribution, the left-hand side in \eqref{2.22}
is also understood in the distributional sense, while the second summand in the right hand side is a Cauchy-integrable function.
Therefore, in view of \eqref{2.17} and taking into account that $\chi(0)=1$, we can write the following equality in the distributional sense
\begin{align}
 &\big[ A (x,\partial_x)P(x-y)]_{pq} \label{2.23-1} 
 =
 \Big[\frac{\partial}{\partial x_k} \Big(a^{pr}_{kj}(x) \frac{\partial P_{rq}(x-y)}{\partial x_j}\Big)\Big] 
 = \Big[\frac{\partial}{\partial x_k} \Big(a^{pr}_{kj}(x) \delta_{rq}\frac{\partial P_\Delta(x-y)}{\partial x_j}\Big)\Big]
\nonumber\\
 &=a^{pq}_{kj}(x)\frac{\partial^2P_\Delta(x-y)}{\partial x_k\,\partial x_j} +
 \frac{\partial a^{pq}_{kj}(x)}{\partial x_k}\,\frac{\partial P_\Delta(x-y)}{ \partial x_j} 
  =
 a^{pq}_{kj}(x)\,\Big[
 \frac{ \delta_{kj}}{3}\;\delta(x-y)+ {\rm{v.p.}} \,
 \frac{\partial^2 P_\Delta(x-y)}{\partial x_k\,\partial x_j}\Big] 
 +
 \frac{\partial a^{pq}_{kj}(x)}{\partial x_k}\,\frac{\partial P_\Delta(x-y)}{ \partial x_j} \nonumber\\
 &=  \boldsymbol\beta_{pq}(x) \;\delta(x-y)+ {\rm{v.p.}}\;[A (x,\partial)P(x-y)]_{pq}, 
\end{align}
where
\begin{align}
\label{2.24}
 \displaystyle
& \boldsymbol\beta(x)
 =[\,\boldsymbol\beta_{pq}(x)\,]_{p,q=1}^3,\quad
 \boldsymbol\beta_{pq}(x)=\frac{1}{3}\; a^{pq}_{kk}(x),
\\
& {\rm{v.p.}}[\,A (x,\partial_x)P(x-y)]_{pq} 
=   {\rm{v.p.}} \Big[-\frac{a^{pq}_{kj}(x)}{4\,\pi} \frac{\partial^2}{\partial x_k\,\partial x_j}\,\frac{1}{|x-y|}\Big]
 +R_{pq}(x,y)
 =   {\rm{v.p.}} \Big[-\frac{a^{pq}_{kj}(y)}{4\,\pi} \frac{\partial^2}{\partial x_k\,\partial x_j}\,\frac{1}{|x-y|}\Big]
 +R^{(1)}_{pq}(x,y)\,, 
  \label{2.23-1-d}\\
& R(x,y)=[ R_{pq}(x,y)]_{p,q=1}^3\,,\;\;\;\; R^{(1)}(x,y)=[ R^{(1)}_{pq}(x,y)]_{p,q=1}^3,
   \label{2.25} 
   \\
&  R_{pq}(x,y):=- \frac{a^{pq}_{kj}(x)}{4 \pi}  \Big\{
  \big[\chi(x-y)-1\big]\, \frac{\partial^2}{\partial x_k \partial x_j}\frac{1}{|x-y|}
  +  \frac{\partial^2 \chi(x-y)}{ \partial x_k\partial x_j}\,\frac{1}{|x-y|}+
  \, \frac{\partial  \chi(x-y)}{\partial x_j}\, \frac{\partial}{ \partial x_k}\frac{1}{|x-y|}
  \nonumber \\
  & \qquad\qquad\quad
   +
  \, \frac{\partial  \chi(x-y)}{\partial x_k}\, \frac{\partial}{ \partial x_j}\frac{1}{|x-y|} \Big\}
   - \frac{1}{4 \pi}\,  \frac{\partial a^{pq}_{kj}(x)}{\partial x_k}\, \Big[
   \frac{\partial  \chi(x-y)}{ \partial x_j}\,\frac{1}{|x-y|}
   +
  \chi(x-y)\, \frac{\partial}{ \partial x_j}\frac{1}{|x-y|} \Big], 
     \label{2.25-1ddd}
   \\[2mm]
&  R_{pq}^{(1)}(x,y):=R_{pq}(x,y)-\frac{a^{pq}_{kj}(x)-a^{pq}_{kj}(y)}{4\,\pi} \frac{\partial^2}{\partial x_k\,\partial x_j}   \frac{1}{|x-y|} \,. 
 \end{align}
Clearly the entries of the matrix-functions $R(x,y)$ and $R^{(1)}(x,y)$ possess  weak singularities of type $ \mathcal{O}(|x-y|^{-2})$ as $x\to y$.

Denote by $\mathring E$   the extension operator by zero from $\Omega$ onto $\Omega^-$.  From the definitions \eqref{2.21} and \eqref{2.21-1d} it is evident that
\begin{align}\label{N0}
\big(\mathcal{N}\,u\big)(y)= \big({\bf N}\mathring E u\big) (y)\;\;\;\text{for}\;\; y\in \Omega,
\quad u\in H^{r}(\Omega),\ r\ge 0.
\end{align}
The definition of $\mathcal N$ can be extended to smaller $r$ as
\begin{align}\label{N1}
\big(\mathcal{N}\,u\big)(y):= \big({\bf N}\widetilde E^r u\big) (y)\;\;\;\text{for}\;\; y\in \Omega,
\quad u\in H^{r}(\Omega),\ -1/2<r<1/2,
\end{align}
where $\widetilde E^r: H^r(\Omega)\to \widetilde H^r(\Omega)$ is the extension operator,  uniquely defined for $-1/2<r<1/2$, see e.g. \cite[Theorem 2.16]{MikJMAA2011}. For $0\le r<1/2$, $\widetilde E^r=\mathring E$ and thus the expressions \eqref{N0} and \eqref{N1} coincide for such $r$.

 From decomposition \eqref{2.23-1-d} it follows that (see, e.g., \cite{BdM}, \cite[Theorem 8.6.1]{HW})
 if $\chi \in X^k$ with integer $k\ge 2$, then
\begin{align}
\label{2.40}
r_{_{\Omega}}{\mathcal N}=r_{_{\Omega}}{\bf N}\,\mathring E \,&:\,H^r(\Omega)\to H^r(\Omega),\qquad\quad  0\le r,\\
\label{2.40-}
r_{_{\Omega}}{\mathcal N}=r_{_{\Omega}}{\bf N}\,\widetilde E^r \,&:\,H^r(\Omega)\to H^r(\Omega),\quad  -1/2<r<1/2,
\end{align}
are bounded since the principal homogeneous symbol of {\bf N} is rational (see \eqref{2.38} in Section~\ref{S-Symb}), and the operators with the kernel functions either $R(x,y)$ or $R_1(x,y)$
maps $H^r(\Omega)$ into $H^{r+1}(\Omega)$  (cf. \cite[ Theorem 5.4]{CMN-Loc2}).
Here and throughout the paper $r_{_{\Omega }}$ denotes the restriction operator to $\Omega$.

Further, by direct calculations one can easily verify that
\begin{align}
\label{2.26}
&
\displaystyle
\lim\limits_{\varepsilon \to 0}
\int _{\Sigma(y,\varepsilon)} P(x-y)\;T(x,\partial_x) u (x)   \,d\Sigma(y,\varepsilon)=0,\\[2mm]
&
\lim\limits_{\varepsilon \to 0}
\int _{\Sigma(y,\varepsilon)}    \;\{T(x,\partial_x)P(x-y)\}\, u (x) \,d\Sigma(y,\varepsilon)
=\Big[\frac{a^{pq}_{kj}(y)}{4\,\pi}\int _{\Sigma_1}    \eta_k\,\eta_j \,d\Sigma_1\Big]_{3\times 3}\,u(y) 
 =\Big[\frac{ a^{pq}_{kj}(y)}{4\,\pi}\,\frac{4\,\pi\,\delta_{kj}}{3}\Big]_{3\times 3}\, u(y) 
=\boldsymbol\beta(y)\,u(y),
\label{2.27}
\end{align}
where  $\Sigma_1$ is a unit sphere, $ \eta=(\eta_1, \eta_2, \eta_3)\in \Sigma_1$  and  $\boldsymbol\beta$ is defined by \eqref{2.24}.

Passing to the limit in \eqref{2.20} as $\varepsilon \to 0$ and using
relations \eqref{2.21}, \eqref{2.26}, and \eqref{2.27}, we obtain
\begin{align}
\label{2.28}
\boldsymbol\beta(y)\,u(y)+  \mathcal{N}\,u(y)-V(T^+u)(y)+W(\gamma^+u)(y) 
= \mathcal{P} \big(A u\big)(y), \;\;\;\; y\in \Omega, 
\end{align}
where $ \mathcal{N}$ is a {\it localized singular integral operator} given by \eqref{2.21}, while   $V$, $W$, and $\mathcal{P}$
are the {\it localized vector single layer, double layer and Newtonian volume potentials},
\begin{align}
\label{2.29}
 &
 V g(y):=-\int _{S} P ( x-y)\, g(x)\,dS_x,  \\
 \label{2.30}
  &
  W g(y):=-\int _{S}\big[\,T (x, \partial_x)\,P(x-y)\,\big]\, \, g(x)\,dS_x, \\
\label{2.31}
&
 \mathcal{P}h(y):= \int _{\Omega} P(x-y)\,h(x)\,dx.
\end{align}
Here the densities $g$ and $h$ are three dimensional vector-functions.
Introducing the following localised scalar Newtonian volume potential
\begin{eqnarray}
\label{2.29-1}
  \mathcal P_\Delta h_0(y):= \int \limits_{\Omega} P_\Delta(x-y)\,h_0(x)\,dx
\end{eqnarray}
with $h_0$ being a scalar density function, we evidently obtain,
$$
 [\mathcal{P}h(y)]_p=\mathcal P_\Delta h_p (y), \quad  p=\overline{1,3},
$$
for any vector function $h=(h_1, h_2,h_3)^\top$.

We will also need the localised vector Newtonian volume potential similar to \eqref{2.31} but with integration over the whole
space ${\mathbb{R}}^3$,
\begin{eqnarray}
\label{2.31-1d}
{\bf P}h(y):= \int _{\mathbb{R}^3} P(x-y)\,h(x)\,dx.
\end{eqnarray}

Mapping  properties of potentials \eqref{2.29}-\eqref{2.31-1d} are investigated in \cite{CMN-Loc2}, \cite{CMN-IEOT2013} and provided in Appendix B.

We refer to relation \eqref{2.28} as {\it Green's third identity}.
Due to the density of $\mathcal D(\overline{\Omega})$ in $H^{1,\,0}(\Omega;A)$ (see \cite[Theorem 3.12]{MikJMAA2011}) and the mapping properties of the potentials,  Green's third
 identity \eqref{2.28} is valid also for $u\in H^{1,\,0}(\Omega;A)$.
 In this case, the co-normal derivative $T^+u$ is understood in the sense of definition \eqref{4-d}.
 In particular,
\eqref{2.28} holds true for solutions of the above formulated Dirichlet BVP \eqref{1}-\eqref{2}.


On the other hand, applying the first Green identity \eqref{1GI} on $\Omega_\varepsilon$ to $u\in H^1(\Omega)$ and to $P(x-y)$, as $v(x)$, and taking the limit as $\varepsilon\to 0$, one can easily derive another, more general form of the third Green identity,
\begin{equation}
\label{2.32}
\boldsymbol\beta(y)\,u(y) + \mathcal{N}\,u(y) + W(\gamma^+ u)(y)=  \mathcal{Q}\,u(y), \;\;\;\; \forall\,y\in \Omega,
\end{equation}
where for the $p$-th component of the vector $ \mathcal{Q}\,u(y)$ we have
\begin{align}
[ \mathcal{Q}\,u(y)]_p: & =-\int _{\Omega} a^{pq}_{kl}(x)\,\frac{\partial P_\Delta(x-y)}{\partial x_k}\;\frac{\partial u_q(x)}{\partial x_l}\,dx 
=
\partial_k\,\mathcal P_\Delta\big( a^{pq}_{kl}\, \partial_l u_q\big)(y)\,, \;\;\;\; \forall\,y\in \Omega.
\label{2.33}
\end{align}

Using the properties of localized potentials described in the Appendix B (see Theorems \ref{tB.1} and \ref{tB.7}) and taking the
trace of equation \eqref{2.28} on $S$ we arrive at the relation for $u\in H^{1,\,0}(\Omega^+;A)$,
\begin{equation}
\label{2.41}
 \displaystyle
 \mathcal N^+{u}-{ \mathcal{V}}(T^+u)+(\boldsymbol\beta-\boldsymbol{\mu})\,\gamma^+ u +{\mathcal{W}}(\gamma^+ u)
 ={ \mathcal{P}}^+ \big(A u\big)\quad   \text{on}\;\;S,
\end{equation}
where the  localized boundary integral operators ${\mathcal{V}}$ and  ${\mathcal{W}}$ are generated by the localized single and
double layer potentials and are defined in \eqref{3.11} and \eqref{3.12},
the matrix $\boldsymbol{\mu}$ is defined by \eqref{mu}, while
\begin{eqnarray*}
 {\mathcal N}^+:=\gamma^+{\mathcal N},\quad
 {\mathcal{P}}^+  := \gamma^+{\mathcal{P}}.
 \end{eqnarray*}
Now we prove the following technical lemma.
\begin{lem}
\label{l2.1}
Let $\chi\in X^3$,
$f \in H^0(\Omega),$
 $F \in H^{1,0}(\Omega,\Delta),$
 $\psi \in H^{-\frac{1}{2}}(S),$  and  $\varphi \in H^{\frac{1}{2}}(S)$.
 Moreover, let $u \in H^1(\Omega)$  and
the following equation hold
\begin{align}
\label{2.44}
\boldsymbol\beta(y)u(y)+ {\mathcal{N} } u(y)-V\psi (y)+W\varphi (y)
=F(y)+\mathcal P f (y), \;\; y\in \Omega.
\end{align}
Then $u\in H^{1,0}(\Omega, A)$.
\end{lem}
\begin{proof} Note that by Theorem \ref{tB.1}, $\mathcal P f \in H^{2}(\Omega)$ for arbitrary $f\in H^0(\Omega)$,
while by Theorem \ref{tB.4} the inclusions  $V\psi , W\varphi \in H^{1,0}(\Omega,\Delta)$ hold for arbitrary $\psi\in H^{-\frac{1}{2}}(S)$  and $\varphi\in H^{\frac{1}{2}}(S)$.
 In view of the relations \eqref{2.32}-\eqref{2.33}
equation \eqref{2.44} can be rewritten component-wise  as
\begin{align}\label{2.44a}
 \partial_k\,{\mathcal P_\Delta}\big( a^{pq}_{kl}\, \partial_l u_q\big)(y)
 =F_p(y)+{\mathcal P_\Delta}f_p(y)+[V\psi (y)]_p
 -
 [W(\varphi-\gamma^+ u)(y)]_p, \quad
  y\in \Omega\;\; p=\overline{1,3}.
\end{align}
By  Theorems \ref{tB.1} and \ref{tB.4} it follows that the right-hand side function
in the equality  belongs to the space
$$
H^{1,0}(\Omega,\Delta):=\{v\in H^{1}(\Omega)\,:\,\Delta v\in H^{0}(\Omega)   \}\,,
$$
since
$\gamma^+ u\in H^{\frac{1}{2}}(S)$, and therefore
\begin{align}
\label{2.46}
\partial_k\,{\mathcal P_\Delta}\big( a^{pq}_{kl}\, \partial_l u_q\big)\in H^{1,0}(\Omega, \Delta)\,.
\end{align}
We have
\begin{align}
\label{2.47}
 \Delta_x\,P_\Delta(x-y)= \delta(x-y)+R_\Delta(x-y),
\end{align}
where
\begin{align}
\label{2.48}
 \displaystyle
 R_\Delta(x-y):= -\frac{1}{4\,\pi}\,\Big\{  \frac{\Delta \,\chi(x-y)}{|x-y|}
 + 2\,\frac{\partial \, \chi(x-y) }{\partial x_l}\, \,\frac{\partial}{ \partial x_l}\frac{1}{|x-y|} \Big\}\,.
 \end{align}
 Clearly, $R_\Delta(x-y)={\mathcal{O}}(|x-y|^{-2})$ as $x\to y$ and by \eqref{2.47}
 and \eqref{2.48} one can establish that
 for arbitrary scalar test function $\phi \in {\mathcal{D}}(\Omega)$, there holds the relation (see, e.g., \cite{MP})
 \begin{align}
\label{2.49}
\Delta \,{\mathcal P_\Delta}\phi (y)=\phi(y) +  {\mathcal{R}}_\Delta\phi (y),\;\;\;y\in \Omega,
\end{align}
where
\begin{align}
\label{2.50}
\displaystyle
{\mathcal{R} }_\Delta\phi (y):= \int _{\Omega} R_\Delta(x-y)\,\phi(x)\,dx.
\end{align}
Evidently \eqref{2.49} remains true also for $\phi\in H^{0}(\Omega)$,  since ${\mathcal{D}}(\Omega)$ is dense in $H^{0}(\Omega)$.
It is easy to see that (see \cite{CMN-Loc2})
\begin{align}
\label{2.50-1}
{\mathcal{R}}_\Delta\,:\, H^{0}(\Omega)\to H^{1}(\Omega)\,.
\end{align}
Consequently,
\begin{align}
 \Delta\Big[\partial_k\,{\mathcal P_\Delta}\big( a^{pq}_{kl}\, \partial_l u_q\big)(y)\,\Big]=
 \partial_k \,\Big[\,\Delta_y{ \mathcal P_\Delta}\big( a^{pq}_{kl}\, \partial_l u_q\big)(y)\,\Big]
  =\partial_k \,\Big[  a^{pq}_{kl}(y)\, \partial_l  u_q(y)  \,\Big]
+\partial_k\,{\mathcal{R} }_\Delta(a^{pq}_{kl}\, \partial_l u_q)(y)&\nonumber\\
 =
 [A\,u(y)]_p+\partial_k\,{\mathcal{R}}_\Delta(a^{pq}_{kl}\, \partial_l u_q)(y)\,,& \;\;\;\;y\in \Omega\,.
 \label{2.51}
 \end{align}
Whence the embedding $A u\in H^{0}(\Omega)$ follows from \eqref{2.44a} due to  \eqref{2.46} and \eqref{2.50-1}.
\end{proof}

Actually,
the continuity of  operator in \eqref{2.50-1} and identity \eqref{2.51}
in the proof of Lemma~\ref{l2.1} imply by \eqref{2.32}
the following assertion.
\begin{cor}\label{C1}
If $\chi\in X^3$ then the following operator
is bounded,
\begin{align*}
{\boldsymbol\beta}+\mathcal{N}\,:\, H^{1,0}(\Omega,A)\to H^{1,0}(\Omega,\Delta)\,.
\end{align*}
\end{cor}

\section{LBDIE formulation of the Dirichlet problem and the equivalence theorem}
\label{ss2.2}
Let $u\in H^{1,0}(\Omega, A)$ be a solution to the Dirichlet BVP \eqref{1}-\eqref{2}
with $\varphi_{_0}\in H^{\frac{1}{2}}(S)$ and  $f\in H^0(\Omega)$.
As we have derived above,   there hold relations \eqref{2.28} and \eqref{2.41}, which now can be rewritten
in the form
\begin{align}
\label{2.52}
&
\displaystyle
( \boldsymbol\beta+\mathcal N){u} -V\psi  =\mathcal{ P}f  -W\varphi_{_0}    \;\;\text{in}\;\; \Omega,\\
\label{2.53}
&
\displaystyle
 \mathcal N^+{u}-\mathcal{ V}\psi =\mathcal{ P}^+ f  -(\boldsymbol\beta-\boldsymbol{\mu})\,\varphi_{_0} -\mathcal{ W}\varphi_{_0}   \;\;\text{on}\;\;S,
\end{align}
where  $\psi:=T^+u\in H^{-\frac{1}{2}}(S)$ and $\boldsymbol{\mu}$ is defined by \eqref{mu}.
One can consider these relations as an LBDIE system with respect to the
unknown vector-functions $u$ and $\psi$.
Now we prove  the following  equivalence theorem.
\begin{thm}
\label{th-D-eq}
Let $ \chi \in X^3_+,  \varphi_0\in  H^{\frac{1}{2}}(S)$  and  $f\in H^0 (\Omega)$.

(i) If a vector-function $u\in H^{1,\,0}(\Omega,A)$ solves the Dirichlet
BVP \eqref{1}-\eqref{2}, then the solution is unique and the pair
$(u,\psi)\in H^{1,\,0}(\Omega,A)\times {H}^{-\frac{1}{2}}(S)$ with
\begin{eqnarray}
\label{2.54}
\psi=T^+u\,,
\end{eqnarray}
solves the LBDIE system \eqref{2.52}-\eqref{2.53}.

(ii) Vice versa, if a pair $(u,\psi)\in H^{1,\,0}(\Omega,A)\times {H}^{-\frac{1}{2}}(S)$
solves the LBDIE system \eqref{2.52}-\eqref{2.53}, then the solution is unique and
the vector-function $u$ solves the Dirichlet BVP \eqref{1}-\eqref{2},
and relation \eqref{2.54} holds.
\end{thm}
\begin{proof} (i) The first part of the theorem is trivial and directly follows form the relations  \eqref{2.28}, \eqref{2.41}, \eqref{2.54}, and  Remark \ref{rBVP}.

(ii) Now, let a pair $(u,\psi)\in H^{1,\,0}(\Omega,A)\times {H}^{-\frac{1}{2}}(S)$ solve
the LBDIE  system  \eqref{2.52}-\eqref{2.53}.
Taking the trace of \eqref{2.52} on  $S$  and  comparing it with  \eqref{2.53}, we get
\begin{eqnarray}
\label{2.55}
\gamma^+ u=\varphi_0 \;\;\text {on}\;\; S.
\end{eqnarray}
Further, since $u \in H^{1,\,0}(\Omega,A)$, we can write  Green's third identity  \eqref{2.28} which in view of \eqref{2.55}
can be rewritten as
\begin{align}
\label{2.56}
&
\displaystyle
(\boldsymbol\beta+\mathcal N){u} -V(T^+u ) =\mathcal{ P}\big(A u\big) -W\varphi_{_0}    \;\;\text{in}\;\; \Omega.
\end{align}
From \eqref{2.52} and \eqref{2.56} it follows that
\begin{eqnarray}
\label{2.57}
V (T^+u-\psi)+ \mathcal{ P} \big(A u-f \big)=0\;\; \text  {in} \;\; \Omega.
\end{eqnarray}
Whence by Lemma 6.3 in \cite{CMN-Loc2} we have
$$
A u=f\;\;\text{in}\;\;\Omega \;\;\;\text{and}\;\;\;
T^+u=\psi\;\;\text{on}\;\;S.
$$
Thus $u$  solves the Dirichlet BVP \eqref{1}-\eqref{2} and equation \eqref{2.54} holds.

The uniqueness of solution to the LBDIE system \eqref{2.52}-\eqref{2.53}
in the space $ H^{1,\,0}(\Omega,A)\times {H}^{-\frac{1}{2}}(S)$
directly follows from the above proved equivalence result and the uniqueness theorem for the Dirichlet problem
\eqref{1}-\eqref{2}, see Remark \ref{rBVP}.
\end{proof}
\section{Symbols and invertibility of a domain operator in the half-space}\label{S-Symb}
In what follows in our analysis, we need the explicit expression of the principal homogeneous symbol matrix
$ {\mathfrak{S}} (\mathcal{N})(y, \xi)$ of the singular integral operator $\mathcal{N}$, which due to
\eqref{2.21}, \eqref{2.21-1d} and \eqref{2.23-1-d}
reads as
\begin{align}
 \big[{\mathfrak{S}} (\mathcal{N})(y, \xi)\big]_{pq}&=
 \big[{\mathfrak{S}} (\mathbf{N})(y, \xi)\big]_{pq}
=
 {\mathcal{F}}_{z\to \xi}\Big[-{\rm{v.p.}}\,\frac{a^{pq}_{kl}(y)}{4\,\pi}
\frac{\partial^2}{\partial z_k\,\partial z_l}\,\frac{1}{|z|}\,\Big]
=-\frac{a^{pq}_{kl}(y)}{4\,\pi}\;{\mathcal{F}}_{z\to \xi}\Big[{\rm{v.p.}} \,\frac{\partial^2}{\partial z_k\,\partial z_l}\,\frac{1}{|z|}\,\Big]
\nonumber\\[2mm]
&
= -\frac{a^{pq}_{kl}(y)}{4\,\pi} \;{\mathcal{F}}_{z\to \xi}\Big[\frac{4\,\pi\, \delta_{kl}}{3}\, \delta(z)
+\frac{\partial^2}{\partial z_k\,\partial z_l}\,\frac{1}{|z|}  \,\Big]
=-\boldsymbol\beta_{pq}(y)- a^{pq}_{kl}(y)(-i\,\xi_k)(-i\,\xi_l)\;{\mathcal{F}}_{z\to \xi}\Big[ \, \frac{1}{4\pi|z|}  \,\Big]
\nonumber\\[2mm]
&
=  \frac{ A_{pq}(y,\xi)}{|\xi|^2}  -\boldsymbol\beta_{pq}(y)\,, 
\quad y\in \overline{\Omega}, \;\;
\xi\in {\mathbb{R}}^3, 
\label{2.34-1}
\end{align}
where
\begin{equation*}
A_{pq} (y,\xi)= a^{pq}_{kl}(y) \,\xi_k \,\xi_l\,, \;\;\;p,q=1,2,3,
\end{equation*}
while the  Fourier transform operator $\mathcal{F}$ is defined as
$$\mathcal{F}g(\xi)={\mathcal{F}}_{z\to \xi}[g(z)]=
\int_{\mathbb{R}^3}g(z)\,e^{i\,z\cdot \xi}\, dz.$$
Here we have applied that ${\mathcal{F}}_{z\to \xi}[(4\pi|z|)^{-1}]=|\xi|^{-2}$ (see, e.g., \cite{Esk}).

As we see the  entries of  principal homogeneous symbol matrix
$ {\mathfrak{S}}(\mathcal{N})(y, \xi)$ of the operator $\mathcal{N}$
are even rational homogeneous functions in $\xi$ of order $0$. It can
easily be verified that both
the characteristic function of the singular kernel in \eqref{2.23-1-d} and
the symbol \eqref{2.34-1} satisfy the Tricomi condition, i.e.,
their integral averages over the unit sphere vanish (cf. \cite{MP}).

Relation \eqref{2.34-1} implies that the principal homogeneous symbols of the singular integral operators
${\bf N}$ and $\boldsymbol\beta+{\bf N}$ read as
\begin{align}
\label{2.38}
&
 {\mathfrak{S}} ( {\bf N})(y, \xi)=  |\xi|^{-2} A(y,\xi)-\boldsymbol\beta \;\;\forall\, y\in \overline{\Omega},
 \; \;\forall\,\xi\in \mathbb{R}^3\setminus\{0\},\\
&
\label{2.39}
 {\mathfrak{S}}( \boldsymbol\beta+{\bf N})(y, \xi)= |\xi|^{-2}  A(y,\xi)
 \;\;\forall\, y\in \overline{\Omega},\;
 \;\;\forall\,\xi\in \mathbb{R}^3\setminus\{0\}.
\end{align}
Due to \eqref{1-d}, the symbol matrix \eqref{2.39}  is positive definite,
\begin{align*}
  [{\mathfrak{S}} ( \boldsymbol\beta+{\bf N})(y, \xi)\, \zeta] \cdot \bar\zeta =
 |\xi|^{-2} \,\bar\zeta\cdot A(y,\xi)\, \zeta  \geq c_1\,|\zeta|^2
 \quad
 \forall\, y\in \overline{\Omega},\;
 \;\;\forall\,\xi\in \mathbb{R}^3\setminus\{0\},
 \;\;\forall\,\zeta\in \mathbb{C}^3, 
\end{align*}
where $c_1$ is the same positive constant as in  \eqref{1-d}.

Denote
\begin{eqnarray*}
\mathbf{B}:=\boldsymbol\beta +{\bf N}.
\end{eqnarray*}
By \eqref{2.39}, the  principal homogeneous symbol matrix of the operator $\mathbf{B}$ reads as
\begin{eqnarray}
\label{ssss}
{\mathfrak{S}} (\mathbf{B})(y,\xi)
= |\xi|^{-2} {A(y,\xi)}
\;\;\;\text{for}\;\;\;y\in \overline{\Omega},\;\;\;\xi\in \mathbb R^3\setminus{\{0\}},
\end{eqnarray}
is an even  rational homogeneous matrix-function of order $0$  in $\xi$ and due to \eqref{1-d} it is positive definite,
\[
[{\mathfrak{S}} (\mathbf{B})(y,\xi)\zeta] \cdot \bar\zeta \geq \,c_1\,|\zeta|^{2}
\;\;\;\text{for all}\;\;\;y\in \overline{\Omega},\;\;\;\xi\in \mathbb R^3\setminus{\{0\}} \text{ and } \zeta\in {\mathbb{C}}^{3}.
 \]
 Consequently, $\mathbf{B}$ is a strongly elliptic
 pseudodifferential operator of zero order
(i.e., Cauchy-type singular integral operator) and the
partial indices of factorization of the symbol \eqref{ssss} equal to zero (cf. \cite{Shar},
\cite{BrSh}, \cite{CD1}).

We need some auxiliary assertions in our further analysis.
 To formulate them, let $\widetilde{y}\in S=\partial \Omega$ be some fixed point
 and consider the frozen symbol
 ${\mathfrak{S}} (\mathbf{\widetilde{B}})(\widetilde{y},\xi)\equiv {\mathfrak{S}} (\mathbf{\widetilde{B}})(\xi)$, where $\mathbf{\widetilde{B}}$ denotes the operator $\mathbf{B}$ written in chosen local co-ordinate system.
 Further, let   $\widehat{\mathbf{\widetilde{B}}}$ denote  the pseudodifferential operator with the symbol
   $$
 {\mathfrak{\widehat{S}}} ({\mathbf{\widetilde{B}}})(\xi^{\,'},\xi_{3})
 :={\mathfrak{S}} ({\mathbf{\widetilde{B}}})\big((1+|\xi^{\,'}|)\omega,\xi_{3}\big),
{\ \rm where\quad } \omega=\frac{\xi^{\,'}}{|\xi^{\,'}|},\quad \xi=(\xi',\xi_3),\quad \xi'=(\xi_1,\xi_2).
  $$

Then the frozen principal homogeneous symbol matrix ${\mathfrak{S}} (\mathbf{\widetilde{B}})( \xi)$ is also the principal homogeneous symbol matrix
 of the operator $\widehat{\mathbf{\widetilde{B}}}$.
{It }can be factorized with respect to the variable $\xi_{3}$ as
\begin{eqnarray}
\label{FF-1}
{\mathfrak{S}} (\mathbf{\widetilde{B}})( \xi)=\mathfrak{S}^{^{(-)}}(\mathbf{\widetilde{B}})( \xi)\,\, \mathfrak{S}^{^{(+)}}(\mathbf{\widetilde{B}})( \xi),
\end{eqnarray}
where
\begin{eqnarray}
\label{FF-1b}
{\mathfrak{S}}^{^{(\pm)}}(\mathbf{\widetilde{B}})( \xi)=
\frac{1}{\Theta^{^{(\pm)}}(\xi^ {'},\xi_3)}\,\widetilde{A}^{^{(\pm)}}(\xi^{\,'},\xi_{3}).
\end{eqnarray}
Here $\Theta^{^{(\pm)}}(\xi^ {'},\xi_3):=\xi_3 \pm i|\xi^{'}| $ are the "plus" and "minus" factors of the symbol $\Theta(\xi):=|\xi|^2$, and $\widetilde{A}^{^{(\pm)}}(\xi^{\,'},\xi_{3})$ are the   "plus" and "minus" polynomial matrix
factors of the first order in $\xi_3$ of the positive definite polynomial
 symbol matrix $\widetilde{A}(\xi^{\,'},\xi_{3})\equiv \widetilde{A}(\widetilde{y,}\,\xi^{\,'},\xi_{3})$ corresponding to the frozen differential operator $A(\widetilde{y},\partial_x)$ at the point $\widetilde{y}\in S$
(see \cite{Ep}, \cite{EJL}, \cite{EpSp}), i.e.
\begin{eqnarray}
\label{Factor-A}
\widetilde{A}(\xi^{'},\xi_{3})=\widetilde{A}^{^{(-)}}(\xi^{\,'},\xi_{3})\,\, \widetilde{A}^{^{(+)}}(\xi^{\,'},\xi_{3})
\end{eqnarray}
with $\det \widetilde{A}^{^{(+)}}(\xi',\tau)\neq 0\,$ for  $\text{Im} \tau >0$\, and
$\det \widetilde{A}^{^{(-)}}(\xi',\tau)\neq 0\,$ for $\text{Im} \tau < 0$. Moreover, the entries of the matrices $\widetilde{A}^{^{(\pm)}}(\xi^{\,'},\xi_{3})$ are homogeneous
functions in $\xi=(\xi',\xi_3)$ of order $1$.

Denote, by $a^{^{(\pm)}}(\xi')$ the coefficients at $\xi_3^3$ in the determinants
$\det \widetilde{A}^{^{(\pm)}}(\xi',\xi_3)$. Evidently,
\begin{eqnarray}
\label{Factor-A-1}
a^{^{(-)}}(\xi')\,a^{^{(+)}}(\xi')=\det \,\widetilde{A}(0,0,1)> 0\;\;\; \text{for} \;\;\;\xi'\neq 0.
\end{eqnarray}

It is easy to see that the factor-matrices $\widetilde{A}^{^{(\pm)}}(\xi',\xi_3)$ have the following structure
\begin{equation*}
\left(\big[\widetilde{A}^{^{(\pm)}}(\xi',\xi_3)\big]^{-1}\right)_{ij} =
\frac{1}{\det \widetilde{A}^{^{(\pm)}}(\xi',\xi_3)}\,
p_{_{ij}}^{^{(\pm)}}(\xi',\xi_3),\quad i,j=1,2,3,
\end{equation*}
where   $p_{_{ij}}^{^{(\pm)}}(\xi',\xi_3)$  are the co-factors of the  matrix  $\widetilde{A}^{^{(\pm)}}(\xi',\xi_3)$, which can be written in the form
\begin{eqnarray}
\label{co-factors}
p_{_{ij}}^{^{(\pm)}}(\xi',\xi_3)=c_{_{ij}}^{^{(\pm)}}(\xi')\,\xi_3^{2}+
b_{_{ij}}^{^{(\pm)}}(\xi')\,\xi_3+d_{_{ij}}^{^{(\pm)}}(\xi').
\end{eqnarray}
Here $c_{_{ij}}^{^{(\pm)}}$,  \;    $b_{_{ij}}^{^{(\pm)}}$ \; and \; $d_{_{ij}}^{^{(\pm)}}$, $i,j=1,2,3$,
are  homogeneous functions in $\xi'$
 of  order $0$, $1$,  and  $2$,  respectively.

 From the above mentioned it follows that the entries of the factor-symbol
 matrices
$
\mathfrak{b}_{kj}^{^{(\pm)}}(\omega, r, \xi_3)
 :={\mathfrak{S}}^{^{(\pm)}}_{kj}(\mathbf{\widetilde{B}})(\xi',\xi_3)
 $,
 ${k,j}= 1,2,3 $,
 with $\omega=\xi'/|\xi'|$ and  $r=|\xi'|$,
 satisfy the following relations:
 \begin{equation}
\label{class-D-1}
 \frac{\partial^l \mathfrak{b}_{kj}^{^{(\pm)}}(\omega, 0, -1)}  {\partial r^l}
 =(-1)^l\, \frac{\partial^l \mathfrak{b}_{kj}^{^{(\pm)}}(\omega, 0, +1)}  {\partial r^l}\,,
 \;\;\;\;l=0,1,2, \dots
 \end{equation}
These relations imply that the entries of the matrices
${\mathfrak{S}}^{^{(\pm)}}(\mathbf{\widetilde{B}})(\xi',\xi_3)$ belong to the class of symbols
 $D_0$ introduced in \cite{Esk}, Ch. III, $\S \,10,$
 \begin{equation}
\label{class-D-2}
 {\mathfrak{S}}^{^{(\pm)}}(\mathbf{\widetilde{B}})(\xi',\xi_3)\in D_0.
 \end{equation}
Denote by   $\Pi^\pm$  the Cauchy type integral operators
\begin{equation}\label{l7a}
\Pi^\pm h(\xi):=\pm \,\frac{i}{2\pi}\;\lim\limits_{t\to 0+}\;
\int_{-\infty}^{+\infty}\frac{h(\xi',\eta_3)\,d\eta_3}{\xi_3\pm i\,t-\eta_3},
\;\;\;
\end{equation}
which are  well defined at any $\xi\in\mathbb R^3$ for a bounded smooth function $h(\xi',\cdot)$ satisfying the relation $h(\xi',\eta_3)=\mathcal{O}(1+|\eta_3|)^{-\kappa}$ with some $\kappa > 0$.

Let
 $\mathring E_{+}$ be  the extension operator by zero from $\mathbb{R}^{3}_{+}$
onto the whole space $\mathbb{R}^{3}$
and $r_+:=r_{_{\mathbb{R}^{3}_{+}}}:  H^{s}(\mathbb{R}^{3})\to H^{s}(\mathbb{R}^{3}_+)$ be the restriction operator to the half-space $\mathbb{R}^{3}_{+}$.
First we prove the following assertion.
\begin{lem}
\label{L3.1}
Let $s\geq 0$ and $\chi \in X^{k}_{+}$ with integer $k\ge 2$.
The operator
    $$
   r_+ \widehat{\mathbf{\widetilde{B}}}\mathring E_{+}\; : \; H^{s}(\mathbb{R}^{3}_{+})\to H^{s}(\mathbb{R}^{3}_{+})
    $$
 is invertible.

Moreover, for $f\in H^{s}(\mathbb{R}^{3}_{+})$, the unique solution of the equation
\begin{eqnarray}
\label{l1}
r_+ \widehat{\mathbf{\widetilde{B}}}\mathring E_{+} u=f
\end{eqnarray}
for $u\in H^{s}(\mathbb{R}^{3}_{+})$ can be represented in the form $ u = r_+u_+$, where
\begin{equation*}
u_+ = \mathring Eu=\mathcal{F}^{-1}\Big\{[{\mathfrak{\widehat{S}}}^{^{(+)}}({\widetilde{\mathbf{B}}})]^{-1}\Pi^{+}
\Big([{\mathfrak{\widehat{S}}}^{^{(-)}}({\widetilde{\mathbf{B}}})]^{-1}\mathcal{F}(f_*)\Big) \Big\}\, ,
\end{equation*}
and $f_*\in H^{s}(\mathbb{R}^{3})$ is an extension of $f\in H^{s}(\mathbb{R}^{3}_+)$ (i.e. $r_+f_*=f$) such that $\|f_*\|_{H^{s}(\mathbb{R}^{3})}=\|f\|_{H^{s}(\mathbb{R}^{3}_+)}$.
\end{lem}
\begin{proof}
First,  we show that if $f\in  H^{0}(\mathbb{R}^{3}_{+})$, then equation \eqref{l1} is uniquely solvable in the space  $H^{0}(\mathbb{R}^{3}_{+})$.
Let $u\in {H}^{0}(\mathbb{R}^{3}_{+})$ be  a solution  of this equation
and let us denote
\begin{eqnarray}\label{l2}
u_{-}:=f_*-\widehat{\mathbf{\mathbf{\widetilde{B}}}}u_{+},
\end{eqnarray}
where $u_{+}:=\mathring E_{+}u\in \widetilde{H}^{0}(\mathbb{R}^{3}_{+})$ and
$f_*\in H^{0}(\mathbb{R}^{3})$  is an arbitrary extension of
$ f\in H^{0}(\mathbb{R}^{3}_{+})$  onto  $\mathbb{R}^{3}_{+}$
such that
$
\|f_*\|_{{H}^{0}(\mathbb{R}^{3})} =\| f\|_{{H}^{0}(\mathbb{R}^{3}_{+})}.
$
Since $f_* \in H^{0}(\mathbb{R}^{3})$  and  $\widehat{\mathbf{\widetilde{B}}}u_{+}\in
 H^{0}(\mathbb{R}^{3})$, we have  $u_{-}\in H^{0}(\mathbb{R}^{3})$.
 In addition, $u_{-}\in \widetilde{H}^{0}(\mathbb{R}^{3}_{-})$.

The Fourier transform of \eqref{l2} leads to the following relation
\begin{eqnarray}\label{l3}
{\mathfrak{\widehat{S}}} ({\mathbf{\widetilde{B}}})(\xi)\mathcal{F}(u_{+})
+\mathcal{F}(u_{-})(\xi)=\mathcal{F}(f_*)(\xi).
\end{eqnarray}
Due to \eqref{FF-1} we have the following factorization
\begin{eqnarray}
\label{FF-2}
{\mathfrak{\widehat{S}}} ({\mathbf{\widetilde{B}}})(\xi^{\,'},\xi_{3})
={\mathfrak{\widehat{S}}}^{^{(-)}}({\mathbf{\widetilde{B}}})(\xi^{\,'},\xi_{3})\,\,
{\mathfrak{\widehat{S}}}^{^{(+)}}({\mathbf{\widetilde{B}}})(\xi^{\,'},\xi_{3}),
\end{eqnarray}
where ${\mathfrak{\widehat{S}}}^{^{(\pm)}}({\mathbf{\widetilde{B}}})(\xi^{\,'},\xi_{3})
={\mathfrak{S}}^{^{(\pm)}}({\mathbf{\widetilde{B}}})\big((1+|\xi^{\,'}|)\omega,\xi_{3}\big)$ with $\omega=\frac{\xi^{\,'}}{|\xi^{\,'}|}$.
Substituting \eqref{FF-2} into \eqref{l3}  and multiplying both sides by
$[{\mathfrak{\widehat{S}}}^{^{(-)}}({\mathbf{\widetilde{B}}})]^{-1}$,  we get
\begin{align}
{\mathfrak{\widehat{S}}}^{^{(+)}}({\mathbf{\widetilde{B}}})(\xi)\,\mathcal{F}(u_{+})(\xi)+
[{\mathfrak{\widehat{S}}}^{^{(-)}}({\mathbf{\widetilde{B}}})(\xi)]^{-1}\,\mathcal{F}(u_{-})(\xi) 
=[{\mathfrak{\widehat{S}}}^{^{(-)}}({\mathbf{\widetilde{B}}})(\xi))]^{-1}\,\mathcal{F}(f_*)(\xi). 
\label{l4}
\end{align}
Introduce the notations
\begin{align}
\label{l5}
v_{+}(x)&=\mathcal{F}^{-1}_{\xi\rightarrow x}\Big({\mathfrak{\widehat{S}}}^{^{(+)}}({\mathbf{\widetilde{B}}})(\xi)\,\mathcal{F}(u_{+})(\xi)\Big),\\
\label{ll5}
v_{-}(x)&=\mathcal{F}^{-1}_{\xi\rightarrow x}\Big([{\mathfrak{\widehat{S}}}^{^{(-)}}({\mathbf{\widetilde{B}}})(\xi)]^{-1}\,\mathcal{F}(u_{-})(\xi)\Big),\\
\label{lll5}
g(x)&=\mathcal{F}^{-1}_{\xi\rightarrow x}\Big([{\mathfrak{\widehat{S}}}^{^{(-)}}({\mathbf{\widetilde{B}}})(\xi)]^{-1}\,\mathcal{F}(f_*)(\xi)\Big).
\end{align}
Then we can conclude that (see \cite{Esk}, Theorem 4.4 and Lemmas 20.2, 20.5)
\begin{equation}
\label{l6-1}
v_{+}\in \widetilde{H}^{0}(\mathbb{R}^{3}_{+}),\quad v_{-}
\in \widetilde{H}^{0}(\mathbb{R}^{3}_{-}),\quad g\in H^{0}(\mathbb{R}^{3}),
\end{equation}
since the degrees of homogeneity of ${\mathfrak{S}}^{^{(+)}}({\mathbf{\widetilde{B}}})(\xi)$  and  ${\mathfrak{S}}^{^{(-)}}({\mathbf{\widetilde{B}}})(\xi)$  equal  to $0$.

In terms of notations \eqref{l5}-\eqref{lll5}, equation \eqref{l4} acquires the form
\begin{eqnarray}
\label{l6}
\mathcal{F}(v_{+})(\xi)+\mathcal{F}(v_{-})(\xi)=\mathcal{F}(g)(\xi).
\end{eqnarray}
In accordance with Lemma 5.4 in \cite{Esk}, we conclude that
the representation of the vector-function $\mathcal{F}(g)(\xi)$  in the form \eqref{l6}
is unique in view of inclusions \eqref{l6-1} which in turn
leads to the following relations
\begin{eqnarray}\label{l7}
\mathcal{F}(v_{+})=\Pi^{+}\mathcal{F}(g),\qquad \mathcal{F}(v_{-})=\Pi^{-}\mathcal{F}(g).
\end{eqnarray}
Now, from  \eqref{l5}, \eqref{lll5} and the first equation in \eqref{l7} it follows that  $u_{+}\in {\widetilde{H}}^{0}(\mathbb{R}^{3}_{+})$    is representable in the form
\begin{equation}
\label{l8}
u_{+}=\mathcal{F}^{-1}\Big\{[{\mathfrak{\widehat{S}}}^{^{(+)}}({\mathbf{\widetilde{B}}})]^{-1}\Pi^{+}
\Big([{\mathfrak{\widehat{S}}}^{^{(-)}}({\mathbf{\widetilde{B}}})]^{-1}\mathcal{F}(f_*)\Big) \Big\}\,.
\end{equation}
Evidently,  for the solution $u \in  H ^{0}(\mathbb{R}^{3}_{+})$ of equation \eqref{l1} then we get the following representation
\begin{equation}
\label{l8-1}
u =r_+ \mathcal{\mathcal{F}}^{-1}\Big\{[{\mathfrak{\widehat{S}}}^{^{(+)}}({\mathbf{\widetilde{B}}})]^{-1}\Pi^{+}
\Big([{\mathfrak{\widehat{S}}}^{^{(-)}}({\mathbf{\widetilde{B}}})]^{-1}\mathcal{F}(f_*)\Big) \Big\}\,.
\end{equation}
Note that the  representation \eqref{l8-1} does not depend on the choice of
the extension $f_*$. Indeed, let $f_{*1} \in {H}^{0}(\mathbb{R}^{3})$
be another extension of $ f \in {H}^{0}(\mathbb{R}^{3}_{+})$, i.e.,
$r_{+}f_{*1}=f$. Since  $f_{-}=f_*-f_{*1} \in
 {\widetilde{H}}^{0}(\mathbb{R}^{3}_{-})$,   it follows that (see \cite{Esk},
Theorem 4.4, Lemmas 20.2 and 20.5)
\begin{eqnarray*}
\mathcal{F}^{-1}\Big([{\mathfrak{\widehat{S}}}^{^{(-)}}({\mathbf{\widetilde{B}}})]^{-1}\mathcal{F}( f_{-})\Big)\in {\widetilde{H}}^{0}(\mathbb{R}^{3}_{-}),
\end{eqnarray*}
while
$$
\Pi^{+}\Big\{[{\mathfrak{\widehat{S}}}^{^{(-)}}({\mathbf{\widetilde{B}}})]^{-1}
\mathcal{F}( f_{-})\Big\}=\mathcal{F}\Big\{\theta^{+} \mathcal{F}^{-1}\Big([{\mathfrak{\widehat{S}}}^{^{(-)}}({\mathbf{\widetilde{B}}})]^{-1}\mathcal{F}( f_{-})\Big) \Big\}=0
$$
(cf. \cite{Esk}, Lemma 5.2). Here $\theta^{+}$ denotes the   multiplication operator by the Heaviside step function $\theta(x_{3})$ that is equal to 1 for $x_{3}>0$
and  vanishes for $x_{3}<0$.
Therefore
$$
\Pi^{+}\Big([{\mathfrak{\widehat{S}}}^{^{(-)}}({\mathbf{\widetilde{B}}})]^{-1}\mathcal{F}(f_*)\Big)=\Pi^{+}\Big([{\mathfrak{\widehat{S}}}^{^{(-)}}({\mathbf{\widetilde{B}}})]^{-1}\mathcal{F}(f_{*1})\Big)
$$
and the claim follows.
If, in particular, $f=0$, then $f_*=0$, and hence $u=0$
by virtue of \eqref{l8}. Thus,  equation \eqref{l1} possesses at most one solution in the space ${H}^{0}(\mathbb{R}^{3}_{+})$.

Further, we show that the function
\begin{equation}
\label{l-9}
u =r_+ \mathcal{F}^{-1}\Big\{[{\mathfrak{\widehat{S}}}^{^{(+)}}({\mathbf{\widetilde{B}}})]^{-1}\Pi^{+}
\Big([{\mathfrak{\widehat{S}}}^{^{(-)}}({\mathbf{\widetilde{B}}})]^{-1}\mathcal{F}(f_*)\Big) \Big\}
\end{equation}
is a solution of equation \eqref{l1} for any  $f\in {H}^{0}(\mathbb{R}^{3}_{+})$.
To this end, let us first note that for
the vector-function under the restriction operator in \eqref{l-9}, the following embedding holds
\begin{equation}
\label{l-9-1}
 \mathcal{F}^{-1}\Big\{[{\mathfrak{\widehat{S}}}^{^{(+)}}({\mathbf{\widetilde{B}}})]^{-1}\Pi^{+}
\Big([{\mathfrak{\widehat{S}}}^{^{(-)}}({\mathbf{\widetilde{B}}})]^{-1}\mathcal{F}(f_*)\Big) \Big\}
\in \widetilde{H}^{0}(\mathbb{R}^{3}_{+}).
\end{equation}
Indeed, by Lemma 5.2 in \cite{Esk}, we have
\begin{eqnarray*}
 \mathcal{F}^{-1}\Big\{[{\mathfrak{\widehat{S}}}^{^{(+)}}({\mathbf{\widetilde{B}}})]^{-1}\Pi^{+}
\Big([{\mathfrak{\widehat{S}}}^{^{(-)}}({\mathbf{\widetilde{B}}})]^{-1}\mathcal{F}(f_*)\Big) \Big\}
 =\mathcal{F}^{-1}\Big\{[{\mathfrak{\widehat{S}}}^{^{(+)}}({\mathbf{\widetilde{B}}})]^{-1}
\mathcal{F}\Big[\theta^+ \mathcal{F}^{-1}
\Big([{\mathfrak{\widehat{S}}}^{^{(-)}}({\mathbf{\widetilde{B}}})]^{-1}\mathcal{F}(f_*)\Big) \Big]\Big\}
 \end{eqnarray*}
 and \eqref{l-9-1} follows from Theorem 4.4, Lemmas 20.2 and 20.5 in \cite{Esk}.
From \eqref{l-9} and \eqref{l-9-1} we obtain
 \begin{equation}
\label{l-9-2dd-1}
 u_+:=\mathring E_{+} u = \mathcal{F}^{-1}\Big\{[{\mathfrak{\widehat{S}}}^{^{(+)}}({\mathbf{\widetilde{B}}})]^{-1}\Pi^{+}
\Big([{\mathfrak{\widehat{S}}}^{^{(-)}}({\mathbf{\widetilde{B}}})]^{-1}\mathcal{F}(f_*)\Big) \Big\}\,.
\end{equation}
By the relation
\begin{gather*}
\Pi^{+}\Big([{\mathfrak{\widehat{S}}}^{^{(-)}}({\mathbf{\widetilde{B}}})]^{-1}\mathcal{F}(f_*)\Big) 
=[{\mathfrak{\widehat{S}}}^{^{(-)}}({\mathbf{\widetilde{B}}})]^{-1}\mathcal{F}(f_*)-\Pi^{-}\Big([{\mathfrak{\widehat{S}}}^{^{(-)}}({\mathbf{\widetilde{B}}})]^{-1}\mathcal{F}(f_*)\Big)
\end{gather*}
(see Lemma 5.4 in \cite{Esk}), we  get from equality \eqref{l-9-2dd-1},
\begin{align*}
{\mathfrak{\widehat{S}}}({\mathbf{\widetilde{B}}})\mathcal{F}(u_{+})& =
{\mathfrak{\widehat{S}}}^{^{(-)}}({\mathbf{\widetilde{B}}})\Pi^{+}
\Big([{\mathfrak{\widehat{S}}}^{^{(-)}}({\mathbf{\widetilde{B}}})]^{-1}\mathcal{F}(f_*)\Big)
 =
\mathcal{F}(f_*)-{\mathfrak{\widehat{S}}}^{^{(-)}}({\mathbf{\widetilde{B}}})
\,\Pi^{-}\Big([{\mathfrak{\widehat{S}}}^{^{(-)}}({\mathbf{\widetilde{B}}})]^{-1}\mathcal{F}(f_*)\Big).
\end{align*}
Since
$$
\mathcal{F}^{-1}\Big\{{\mathfrak{\widehat{S}}}^{^{(-)}}({\mathbf{\widetilde{B}}})\,\Pi^{-}
\Big([{\mathfrak{\widehat{S}}}^{^{(-)}}({\mathbf{\widetilde{B}}})]^{-1}\mathcal{F}(f_*)\Big)\Big\}\in {\widetilde{H}}^{0}(\mathbb{R}^{3}_{-}),
$$
(cf. \cite{Esk}, Theorems 4.4, 5.1, Lemmas 20.2, 20.5 ),  we easily derive
\begin{align*}
r_+\widehat{\mathbf{\widetilde{B}}}\,u_+ & =r_+(f_*)-r_+\mathcal{F}^{-1}\Big\{{\mathfrak{\widehat{S}}}^{^{(-)}}({\mathbf{\widetilde{B}}})
\,\Pi^{-}
\Big([{\mathfrak{\widehat{S}}}^{^{(-)}}({\mathbf{\widetilde{B}}})]^{-1}\mathcal{F}(f_*)\Big)\Big\}
 =
r_+(f_*)=f,
\end{align*}
i.e., the vector-function \eqref{l-9} solves equation \eqref{l1} and belongs to the space ${H}^{0}(\mathbb{R}^{3}_{+})$ for $f\in {H}^{0}(\mathbb{R}^{3}_+)$.

In what follows we prove that for $f\in {H}^{s}(\mathbb{R}^{3}_+)$ and
 $f_*\in {H}^{s}(\mathbb{R}^{3})$
such that
\begin{eqnarray}
\label{10p}
\|f_*\|_{{H}^{s}(\mathbb{R}^{3})} =\| f\|_{{H}^{s}(\mathbb{R}^{3}_{+})}
\;\;\text{for}\;\;\; s\geq 0,
\end{eqnarray}
the vector-function defined by \eqref{l-9} satisfies the inequality
\begin{eqnarray}
\label{10pp}
\|u \|_{{H}^{s}(\mathbb{R}^{3}_{+})}\leq C\,\,\| f\|_{{H}^{s}(\mathbb{R}^{3}_{+})},
\end{eqnarray}
and hence belongs to  ${H}^{s}(\mathbb{R}^{3}_{+})$.
Indeed, since   by Lemma 5.2 and Theorem 5.1 in \cite{Esk} 
$$
\Pi^{+}(\mathcal{F}g)=\mathcal{F}(\theta^{+} g) \;\; \text{for all} \;\;g\in  H^0(\mathbb{R}^3) ,
$$
then representation \eqref{l-9-2dd-1} of  $u_{+}$  can be rewritten as
\begin{eqnarray*}
u_{+}=\mathcal{F}^{-1}\Big\{[{\mathfrak{\widehat{S}}}^{^{(+)}}
({\mathbf{\widetilde{B}}})]^{-1}\mathcal{F}\Big[\theta^{+}\mathcal{F}^{-1}
\Big([{\mathfrak{\widehat{S}}}^{^{(-)}}({\mathbf{\widetilde{B}}})]^{-1}\mathcal{F}(f_*)\Big)\Big] \Big\}.
\end{eqnarray*}

Therefore, using  \eqref{10p} and  in view of \eqref{class-D-2}, from Theorem 10.1, Lemmas 4.4, 20.2, and 20.5  in \cite{Esk} we finally derive
\begin{align*}
\|u \|_{{H}^{s}(\mathbb{R}^{3}_{+})}&\leq c_1 \,\,\left\|\mathcal{F}^{-1}\Big([{\mathfrak{\widehat{S}}}^{^{(-)}}
({\mathbf{\widetilde{B}}})]^{-1}\mathcal{F}(f_*)\Big)\right \|_{{H}^{s}(\mathbb{R}^{3}_{+})}
\leq c_1\,\,\left\|\mathcal{F}^{-1}\Big([{\mathfrak{\widehat{S}}}^{^{(-)}}
({\mathbf{\widetilde{B}}})]^{-1}\mathcal{F}(f_*)\Big)\right \|_{{H}^{s}(\mathbb{R}^{3})}
\leq c_2 \,
\|f_* \|_{{H}^{s}(\mathbb{R}^{3})} = c_2\,\| f \|_{{H}^{s}(\mathbb{R}^{3}_{+})}
\end{align*}
with some positive constants $c_1$ and $c_2$, whence \eqref{10pp}  follows.
\end{proof}

\begin{lem}
\label{L3.2}
Let the factor matrix $\widetilde{A}^{^{(+)}}(\xi',\tau)$ be as in \eqref{Factor-A}, and
 $a^{^{(+)}}$ and  $c_{_{ij}}^{^{(+)}}$ be as in \eqref{Factor-A-1} and \eqref{co-factors}
 respectively.
Then the following equality holds
\begin{eqnarray}
\label{cont-1}
\frac{1}{2\pi i} \int_{\Gamma^-}\big[\widetilde{A}^{^{(+)}}(\xi',\tau)\big]^{-1}d\tau=
\frac{1}{a^{^{(+)}}(\xi')}C^{^{(+)}}(\xi')\,,
\end{eqnarray}
where $ C^{^{(+)}}(\xi')=\big[\, c_{_{ij}}^{^{(+)}}(\xi')\,\big]_{ij=1}^3$ and
$\det \,[\,C^{^{(+)}}(\xi')\,]\neq 0$ for $\xi'\neq 0.$
Here $\Gamma^{-}$ is a contour in the lower complex half-plane enclosing all the roots of the polynomial $\det \widetilde{A}^{^{(+)}}(\xi',\tau)$ with respect to $\tau$.
\end{lem}
\begin{proof}
Note that $\det \widetilde{A}^{^{(+)}}(\xi',\tau)$ is a third order polynomial  in
$\tau$, while $p^{^{(+)}}_{_{ij}}(\xi',\tau)$ is a second order polynomial in $\tau$ defined
in \eqref{co-factors}.

Let \,$\Gamma_{_{R}}$\, be a circle centred at the origin and having sufficiently large radius $R$.
By the Cauchy theorem then we derive
\begin{align}\label{L2}
\frac{1}{2\pi i} \int_{\Gamma^-}\big\{\big[\widetilde{A}^{^{(+)}}(\xi',\tau)\big]^{-1}\big\}_{ij}
\,d\tau
=&\frac{1}{2\pi i} \int_{\Gamma^-}\frac{p_{_{ij}}^{^{(+)}}(\xi',\tau)
}
{\det \widetilde{A}^{^{(+)}}(\xi',\tau)}\,
\,d\tau 
=\frac{1}{2\pi i} \int_{\Gamma_R}\frac{p_{_{ij}}^{^{(+)}}(\xi',\tau)
}
{\det \widetilde{A}^{^{(+)}}(\xi',\tau)}\,
\,d\tau\nonumber \\
=&\frac{1}{2\pi i} \;\frac{c_{_{ij}}^{^{(+)}}(\xi')}{a^{^{(+)}}(\xi')}\int_{\Gamma_R}\frac{1
}
{\tau}\,\,d\tau 
 + \int_{\Gamma_R}Q_{_{ij}}(\xi',\tau)\,\,d\tau
 =\frac{c_{_{ij}}^{^{(+)}}(\xi')}{a^{^{(+)}}(\xi')} + \int_{\Gamma_R}Q_{_{ij}}(\xi',\tau)\,\,d\tau,
\end{align}
where \;\;\;$Q_{_{ij}}(\xi',\tau)=O(|\tau|^{-2})$  \;\;\;as\;\;\;$|\tau|\rightarrow\infty$.\\
It is clear that
$$
\lim\limits_{R\to \infty}\int_{\Gamma_R}Q_{_{ij}}(\xi',\tau)\,\,d\tau=0.
$$
Therefore by passing to the limit in \eqref{L2} as $R\rightarrow\infty$  we obtain
$$
\frac{1}{2\pi i} \int_{\Gamma^-}\big\{\big[\widetilde{A}^{^{(+)}}(\xi',\tau)\big]^{-1}\big\}_{ij}
\,d\tau=\frac{c_{_{ij}}^{^{(+)}}(\xi')}{a^{^{(+)}}(\xi')}.
$$

Now we show that  $\det [C^{^{(+)}}]\neq 0$.  We introduce the notations
\begin{equation*}
P^{^{(+)}}(\xi',\xi_3)=[p_{_{ij}}^{^{(+)}}(\xi',\xi_3)]_{ij=1}^3=
C^{^{(+)}}(\xi') \xi_3^2+B^{^{(+)}}(\xi') \xi_3+D^{^{(+)}}(\xi'),
\end{equation*}
where
\begin{gather*}
 B^{^{(+)}}(\xi')=[b_{_{ij}}^{^{(+)}}(\xi')]_{ij=1}^3 \
   \text{and}\; \;
 D^{^{(+)}}(\xi')=[d_{_{ij}}^{^{(+)}}(\xi')]_{ij=1}^3.
\end{gather*}
Since   $\det[\widetilde{A}^{^{(+)}}(\xi',\xi_3)]^{-1}\neq0$  for  $\xi=(\xi',\xi_3)\neq 0$,
therefore
$\det P^{^{(+)}}(\xi',\xi_3)\neq 0$   for $\xi=(\xi',\xi_3)\neq0$.

Let  us introduce new coordinates $r=|\xi'|$, $\omega=\xi'/|\xi'|$   and  denote
\begin{eqnarray*}
 {\mathcal{P}}^{^{(+)}}(\omega,r,\xi_3):=P^{^{(+)}}(\xi',\xi_3)=P^{^{(+)}}(\omega\, r,\xi_3).
\end{eqnarray*}
Then we have
\begin{gather}
\label{p1}
\det {\mathcal{P}}^{^{(+)}}(\omega,r,\xi_3)
=\det P^{^{(+)}}(\xi',\xi_3)
=\det \big( C^{^{(+)}}(\omega) \xi_3^2+B^{^{(+)}}(\omega)\, \xi_3 \,r +D^{^{(+)}}(\omega) r^2\big)\neq 0 \;\;
\text{for all}\;\xi_3\neq0.
\end{gather}
Whence
\begin{eqnarray*}
\lim\limits_{r\to 0} \det {\mathcal{P}}^{^{(+)}}(\omega,r,\xi_3)=
\xi_3^6\,\det C^{^{(+)}}(\omega)\, ,
\end{eqnarray*}
 consequently  $\det C^{^{(+)}}(\omega)\neq 0$
and Lemma \ref{L3.2} is proved.
\end{proof}

For further use, let us introduce the auxiliary operator $\Pi'$ defined as
\begin{align*}
\Pi'(g)(\xi') & :=
\lim\limits_{x_3\to 0+}
r_{_{\!\mathbb{R}^{3}_{+}}}
\mathcal{F}^{-1}_{\xi_3\to x_3}[g(\xi',\xi_3)] 
 =
\frac{1}{2\pi}\lim\limits_{x_3\to 0+}
\int_{-\infty}^{+\infty}g(\xi',\xi_3)e^{-ix_3\xi_3}\,d\xi_3 
=\frac{1}{2\pi}\int_{-\infty}^{+\infty}g(\xi',\xi_3)\,d\xi_3\;\;\text{for}\;\;g(\xi', \cdot)\in L_1(\mathbb{R}). 
\end{align*}
The operator $\Pi'$ can be extended to the class of functions
$g(\xi',\xi_3)$ that are rational in $\xi_3$ with the denominator not vanishing
for real non-zero $\xi=(\xi',\xi_3)\in \mathbb{R}^3\setminus\{0\}$,
homogeneous of order $m\in \mathbb{Z}:=\{0, \pm 1, \pm 2, \dots\}$
in $\xi$ and  infinitely differentiable with respect to $\xi$ for $\xi'\neq 0$.
Then one can show  that (cf. Appendix C in \cite{CMN-IEOT2013} )
\begin{equation}
\label{d1d1d1-N-3}
\Pi'(g)(\xi')=\lim\limits_{x_3\to 0+}
r_{_{\!\mathbb{R}_{+}}}
\mathcal{F}^{-1}_{\xi_3\to x_3}[g(\xi',\xi_3)]=-\frac{1}{2\pi}\int_{\Gamma^-}g(\xi',\zeta)\,d\zeta,
 \end{equation}
where $r_{_{\!\mathbb{R}_{+}}}$ denotes the restriction operator onto $\mathbb{R}_{+}=(0,\,+\infty)$
with respect to $x_3$, $\Gamma^-$ is a contour in the lower complex half-plane in $\zeta$,
orientated anticlockwise and enclosing
all the poles of the rational function  $g(\xi',\,\cdot)$.
It is clear that if  $g(\xi',\zeta)$  is holomorphic in $\zeta$ in the lower complex half-plane
  (Im$\,\zeta<0)$, then $\Pi'(g)(\xi')=0$.

\section{Invertibility of the Dirichlet LBDIO}

From Theorem \ref{th-D-eq} it follows that the LBDIE system
\eqref{2.52}-\eqref{2.53}, which has a special right hand side, is uniquely solvable in
 the space $ H^{1,\,0}(\Omega,A)\times {H}^{-1/2}(S)$.
Let us investigate the localized boundary-domain integral
operator, generated by the left hand side expressions in \eqref{2.52}-\eqref{2.53},
in appropriate  functional spaces.

 The LBDIE system  \eqref{2.52}-\eqref{2.53} with an arbitrary right hand side
 vector-functions from the space
  $ H^{1}(\Omega)\times {H}^{1/2}(S)$ can be written as
\begin{eqnarray}
\label{14-1}
&&
\mathbf{B}\mathring E{u}\, -V\psi=F_1\;\;\text{in}\;\;\Omega,\\
&&
\label{15}
{\bf N} ^{+} \mathring E{u}  -\mathcal{V}\psi=F_2\;\;\text{on}\;\;S,
\end{eqnarray}
where $\mathbf{B}=\boldsymbol\beta +{\bf N}$, $F_1\in H^1(\Omega)$ and $F_2\in H^{1/2}(S)$.
Let us denote by $\mathfrak{D}$ the localized boundary-domain integral operator generated by the left hand side expressions
in  LBDIE system \eqref{14-1}-\eqref{15},
\begin{equation*}
 \mathfrak{D}:=
\left[
\begin{array}{ll}
r_{_{\!\Omega }}\mathbf{B}\mathring E &   \;\;\; -r_{_{\!\Omega }}V  \\
 {\bf N}^{+}\mathring E                                     &    \;\;\; -\mathcal{V}    \\
\end{array}
\right].
\end{equation*}

We would like to prove the following assertion.
\begin{thm}
\label{th2.2}
Let the localising function $\chi\in X_+^\infty$ and  $r > -\frac{1}{2}$.   Then the operator
\begin{eqnarray}
\label{15-ddd}
\mathfrak{D}\;: H^{r+1}(\Omega)\times H^{r-1/2}(S)
\to
H^{r+1}(\Omega)\times {H}^{r+1/2}(S)
\end{eqnarray}
is invertible.
\end{thm}
We will reduce the theorem proof to several lemmas.

\begin{lem}\label{Step 1} Let $\chi\in X^\infty$. The operator  $r_{_{\!\Omega }}\mathbf{B}\mathring E \,:\,H^{s}(\Omega)\to H^{s}(\Omega)$ for $s\geq 0$  is Fredholm with zero index.
\end{lem}
\begin{proof}
Since \eqref{ssss} is a rational function in $\xi$,  we can apply the theory of pseudodifferential operators with symbol satisfying the transmission conditions
 (see \cite{Esk},   \cite{BdM}, \cite{RS}, \cite{Shar}, \cite{BrSh}).
Now with the help of the local principle (see Lemma 23.9 in \cite{Esk}) and Lemma \ref{L3.1}
we deduce that the operator
    $$
  \mathcal{B}:=  r_{_{\Omega }} \mathbf{B}\,\mathring E\; : \; H^{s}(\Omega)\to H^{s}(\Omega)
    $$
 is Fredholm for all $s\geq 0$.\\
To show that Ind$\,\mathcal{B}=0$, we use  that the operators  $\mathcal{B}$  and
$$
\mathcal{B}_t=r_{_{\!\Omega }}(\boldsymbol\beta +t\, {\bf N})\mathring E,
$$
where  $t\in [0,1]$, are homotopic.  Note that $\mathcal{B}=\mathcal{B}_1$. The principal homogeneous symbol of the operator  $\mathcal{B}_{t}$ has the form
\begin{equation*}
{\mathfrak{S}} (\mathcal{B}_{t})(y,\xi)=\boldsymbol\beta(y) +t\;{\mathfrak{S}} ({\bf N})(y,\xi)=(1-t)\boldsymbol\beta(y)+t {\mathfrak{S}} (\mathbf{B})(y,\xi).
\end{equation*}
It is easy to see that the symbol ${\mathfrak{S}} (\mathcal{B}_{t})(y,\xi)$ is positive definite,
\begin{equation*}
[{\mathfrak{S}} (\mathcal{B}_{t})(y,\xi)\zeta] \cdot \bar\zeta =
 (1-t)[\boldsymbol\beta(y)\,\zeta] \cdot \bar\zeta + t [{\mathfrak{S}} (\mathbf{B})(y,\xi)\zeta] \cdot \bar\zeta \geq c
|\zeta|^{2}
\end{equation*}
for all $y\in \overline{\Omega}$,
$\xi\neq 0$, $\zeta \in\mathbb{C}^{3}$ and  $t\in [0,1]$,  where $c$  is some positive number.\\
Since ${\mathfrak{S}} (\mathcal{B}_{t})(y,\xi)$ is rational, even, and homogeneous of order zero in $\xi$,
we conclude, as above, that the operator
$$
\mathcal{B}_{t} \; : \; H^{s}(\Omega)\to H^{s}(\Omega)
$$
is Fredholm for all $s\geq0$ and for all $t\in[0,1]$. Therefore Ind$\,\mathcal{B}_t$ is the same for all $t\in[0,1]$.
On the other hand, due to the equality $\mathcal{B}_0=r_{_{\!\Omega }}\,I$, we get
$$
  {\rm{Ind}} \,\mathcal{B}={\rm{Ind}}\, \mathcal{B}_1= {\rm{Ind}}\, \mathcal{B}_t=\, {\rm{Ind}}\, \mathcal{B}_0=0.
$$
\end{proof}

\begin{lem}\label{Step 2}
Let $\chi\in X^\infty$. The operator $ \mathfrak{D}$ given by \eqref{15-ddd} is Fredholm.
\end{lem}
\begin{proof}
To investigate  Fredholm properties of the operator $\mathfrak{D}$
we apply the local principle (cf. e.g., \cite{Agr}, \cite{Esk}, $\S\, 19$ and  $\S\, 22$).
Due to this principle, we have to
show first that the operator $\mathfrak{D}$ is locally Fredholm at an arbitrary "frozen" interior point
$\widetilde{y}\in \Omega$, and secondly
 that the so called generalized {\it \v{S}apiro-Lopatinski\u{i}
condition}  for the operator $\mathfrak{D}$ holds at an arbitrary "frozen" boundary point $\widetilde{y} \in S$.
To obtain the explicit form of this condition we proceed as follows.
Let $\widetilde{U}$  be a  neighbourhood  of a fixed point  $\widetilde{y}\in \overline{\Omega}$
and let   $\widetilde{\psi}_0, \widetilde{\varphi}_0 \in \mathcal{D}(\widetilde{U})$ such that
$$
{\rm{supp}}\,\widetilde{\psi}_0\cap {\rm{supp}}\,\widetilde{\varphi}_0 \neq \emptyset,
 \quad \widetilde{y}\in {\rm{supp}}\,\widetilde{\psi}_0\cap {\rm{supp}}\, \widetilde{\varphi}_0,
 $$
 and consider the operator $\widetilde{\psi}_0\mathfrak{D}\,\widetilde{\varphi}_0.$
We consider separately two possible cases, case (1): $  \widetilde{y}\in {\Omega}$, and   case (2): $\widetilde{y }\in S$.

 {\it Case (1)}. If $\widetilde{y }\in {\Omega}$ then we can choose a
 neighbourhood   ${\widetilde{U}}$ such that   $\overline{\widetilde{U}}\subset{\Omega}$.
  Therefore the operator $\widetilde{\psi}_0\mathfrak{D}\,\widetilde{\varphi}_0$ has the same Fredholm properties as the operator
  $\widetilde{\psi}_0{\mathbf{B}}\,\widetilde{\varphi}_0$ (see the similar arguments in the proof of Theorem 22.1 in \cite{Esk}).
  Then by Lemma~\ref{Step 1} we conclude that $\widetilde{\psi}_0\mathfrak{D}\,\widetilde{\varphi}_0$ is a locally Fredholm operator
  at interior points  of $\Omega$.

{\it Case (2)}. If   $\widetilde{y} \in S$,  then at this point we have to "freeze" the operator $\widetilde{\psi}_0 \,\mathfrak{D}\,\widetilde{\varphi}_0 $,  which means that we can choose
a neighbourhood ${\widetilde{U}}$ sufficiently small such that at the  local
co-ordinate system with the origin  at the point $\widetilde{y}$ and the third axis coinciding with the normal
vector at the point $\widetilde{y} \in S$, the following decomposition holds
\begin{eqnarray}
\label{o-15}
\widetilde{\psi}_0 \mathfrak{D}\,\widetilde{\varphi}_0 =\widetilde{\psi}_0 \,
\Big(\widehat{\mathfrak{\widetilde{D}}} + \mathbf{\widetilde{K}} +\mathbf{\widetilde{T}}\Big)\,\widetilde{\varphi}_0,
\end{eqnarray}
where
\begin{eqnarray*}
\mathbf{\widetilde{K}} \;: H^{r+1}(\mathbb{R}^3_{+})\times H^{r-1/2}(\mathbb{R}^2)
\to
H^{r+1}(\mathbb{R}^3_{+})\times {H}^{r+1/2}(\mathbb{R}^2)
\end{eqnarray*}
  is a bounded operator with small norm, while
\begin{eqnarray*}
\mathbf{\widetilde{T}}\;: H^{r+1}(\mathbb{R}^3_{+})\times H^{r-1/2}(\mathbb{R}^2)
\to
H^{r+2}(\mathbb{R}^3_{+})\times {H}^{r+3/2}(\mathbb{R}^2)
\end{eqnarray*}
is a bounded operator.
The operator
\begin{equation*}
\widehat{\mathfrak{\widetilde{D}}}:=
\left[
\begin{array}{ll}
r_+\widehat{\mathbf{\widetilde{B}}} \mathring E&   \;\;\; -r_+\widehat{\widetilde{V}}  \\[1ex]
 \widehat{{\mathbf{\widetilde N}}}\vphantom{\mathbf N}^{+}\mathring E                                   &    \;\;\; -\widehat{\mathcal{\widetilde{V}}}\\
\end{array}
\right]
\end{equation*}
with $r_+=r_{_{\!\mathbb{R}^{3}_{+}}}$, is  defined in the upper half-space $\mathbb{R}^3_{+}$
and possesses the following mapping property
\begin{equation}
\label{ddd-1}
\widehat{\mathfrak{\widetilde{D}}}\;: H^{r+1}(\mathbb{R}^3_{+})\times H^{r-1/2}(\mathbb{R}^2)
\to
H^{r+1}(\mathbb{R}^3_{+})\times {H}^{r+1/2}(\mathbb{R}^2).
\end{equation}
The operators involved in the expression of $\widehat{\mathfrak{\widetilde{D}}}$ are defined as follows: for the operator $\widetilde{M}$, the operator $\widehat{\widetilde{M}}$ denotes  the operator in $\mathbb{R}^n$ $(n=2,3)$ constructed by the symbol
$$
\widehat{\mathfrak{S}} ({\widetilde{M}})(\xi )={\mathfrak{S}} {}(\widetilde{M})
\Big((1+|\xi '|)\omega, \xi_3\Big)  \;\;\text{if} \;\; n=3
$$
and
$$
\widehat{\mathfrak{S}} ({\widetilde{M}})(\xi )={\mathfrak{S}} {}(\widetilde{M})
\Big((1+|\xi '|)\omega\Big) \;\;\; \text{if} \;\; \;n=2,
$$
where  \,$\omega=\frac{\xi '}{| \xi '|}$, \, $\xi=(\xi ',\xi_n)$, \,$\xi '=(\xi_1,...,\xi_{n-1})$.

The generalized \v{S}apiro-Lopatinski\u{i} condition is related to the invertibility of the operator
\eqref{ddd-1}. Indeed,  let us write the system corresponding to the operator $\widehat{\mathfrak{\widetilde{D}}}$:
\begin{eqnarray}
\label{14-1-dddd}
r_+\widehat{\mathbf{\widetilde{B}}} \mathring E {\widetilde{u}}\, -r_+\widehat{\widetilde{V}} \widetilde{\psi} =\widetilde{F}_1\;\;&\text{in}&\;\;\mathbb{R}^3_{+}\,,\\
\label{15-dddd}
\widehat{{\mathbf{\widetilde N}}}\vphantom{\mathbf N}^{+}\mathring E{\widetilde{u}}-\widehat{\mathcal{\widetilde{V}}} \widetilde{\psi } =\widetilde{F}_2\;\;&\text{on}&\;\;\mathbb{R}^2\,,
\end{eqnarray}
where       $\widetilde{F}_1\in H^1(\mathbb{R}^3_{+})$,  $\widetilde{F}_2 \in H^{1/2}(\mathbb{R}^2 )$.\\
Note that the operator $r_+\widehat{\mathbf{\widetilde{B}}} \mathring E$ is a singular integral operator with even rational elliptic principal
homogeneous symbol. Then  due to Lemma \ref{L3.1}  the operator
\begin{eqnarray*}
r_+\widehat{\mathbf{\widetilde{B}}} \mathring E \;: H^{r+1}(\mathbb{R}^3_{+})
\to
H^{r+1}(\mathbb{R}^3_{+})
\end{eqnarray*}
is invertible,
we can determine $\widetilde{u}$ from equation \eqref{14-1-dddd} and write
\begin{gather}
\label{ddd-11}
\mathring E {\widetilde{u}}
={\mathring E}\big[r_+ \widehat{\mathbf{\widetilde{B}}}\mathring E\big]^{-1}\widetilde{f}
=
 \mathcal{F}^{-1}\Big\{\big[\widehat{\mathfrak{S}}^{^{(+)}}({\mathbf{\widetilde{B}}})\big]^{-1}\;
\Pi^+\Big(\big[\widehat{\mathfrak{S}}^{^{(-)}}({\mathbf{\widetilde{B}}})\big]^{-1}\mathcal{F}(\widetilde{f}_*)
\Big)\Big\}\,, 
\end{gather}
where
$\widetilde{f}_* =\widetilde{F}_{1*} +\widehat{\widetilde{V}} \widetilde{\psi}$  is an extension of $\widetilde{f} =\widetilde{F}_1 +r_+\widehat{\widetilde{V}} \widetilde{\psi}$ from $\mathbb{R}^{3}_{+}$ to $\mathbb{R}^{3}$
preserving the function space.
The symbols ${\widehat{\mathfrak{S}}}^{^{(\pm)}}({M})$ denote the so called "plus" and "minus"
factors in the factorization of the symbol
$\widehat{{\mathfrak{S}}}{} ( {M})$ with respect to the variable $\xi_3$.
Note that the function $\mathring E{\widetilde{u}}$ in \eqref{ddd-11} does not depend on the chosen extension $\widetilde{f}_*$ of $\widetilde{f}$.

Substituting \eqref{ddd-11} into \eqref{15-dddd} leads to the following pseudodifferential
  equation with respect to the unknown function  $\widetilde{\psi}$:
\begin{gather}
\widehat{{\mathbf{\widetilde N}}}\vphantom{\mathbf N}^{+}
\mathcal{F}^{-1}\Big\{\big[\widehat{{\mathfrak{S}}}^{^{(+)}}( {\mathbf{\widetilde{B}}})]^{-1}\;
\Pi^+
\Big(\big[\widehat{\mathfrak{S}}^{^{(-)}}({\mathbf{\widetilde{B}}})\big]^{-1}\mathcal{F}
(\widehat{\widetilde{V}} \widetilde{\psi})
\Big)\Big\} 
-\widehat{\mathcal{\widetilde{V}}} \widetilde{\psi}
=\widetilde{F}\;\;\text{on}\;\;\mathbb{R}^2, 
\label{o-1}
\end{gather}
where
$$
\widetilde{F }=\widetilde{F}_2 - \widehat{{\mathbf{\widetilde N}}}\vphantom{\mathbf N}^{+}\mathring E\,
\big[r_+ \widehat{\mathbf{\widetilde{B}}}\mathring E\big]^{-1}\widetilde{F}_1.
$$
It is easy to see that
\begin{align*}
 {\mathbf{\widetilde N}}^{+} \, {v}\,(\widetilde{y}\, ')&
 =\Big[
\mathcal{F}^{-1}_{ \xi \to \widetilde{y} }\big[
{\mathfrak{S}} ({\mathbf{\widetilde N}})( \xi )\;
\mathcal{F}(v)(\xi )\big]\Big]_{\widetilde{y} _{_{3}}=0+} 
=\mathcal{F}^{-1}_{{\xi}\,'\to \widetilde{y}\, '}\Big[\Pi'\big[
{\mathfrak{S}} ({\mathbf{\widetilde N}})\;\mathcal{F}(v)\big](\xi')\Big].
\end{align*}
In view of the relation (see, e.g., \cite[Eq. (4.1)]{Costabel1988}, \cite[Eqs. (B.5), (B.6)]{CMN-IEOT2013})
\begin{align*}
\widetilde{V}\widetilde{\psi}(y)
=-\langle\gamma\widetilde{P}(\cdot-y),\widetilde{\psi}\rangle_S
=-\langle\widetilde{P}(\cdot-y),\gamma^*\widetilde{\psi}\rangle_{\mathbb R^3}
=-\widetilde{\mathbf{P}}(\gamma^*\widetilde{\psi})(y),
\end{align*}
where the operator $\gamma^*$ is dual to the trace operator $\gamma$. When the
surface $S$ coincides with $\mathbb{R}^2=\partial\mathbb R^3_+$, then we
have $\gamma^*\widetilde{\psi}=\widetilde\psi(\widetilde y')\otimes\delta_3$
with $\delta_3$ being the one-dimensional Dirac distribution in the $\widetilde{y}_3$
 direction. Then  we arrive at the equality
 \begin{multline*}
\widehat{\mathbf{\widetilde N}}\vphantom{\mathbf N}^{+}\mathcal{F}^{-1}_{\xi \to \widetilde{x}} \Big\{\big[\widehat{\mathfrak{S}}^{^{(+)}}(\mathbf{\widetilde{B}})( \xi )\big]^{-1}\;
\Pi^+\Big(\big[\widehat{\mathfrak{S}}^{^{(-)}}({\mathbf{\widetilde{B}}})\big]^{-1}\mathcal{F}(\widehat{\widetilde{V}}\widetilde{\psi})\Big) ( \xi ) \Big\}\,(\widetilde{y}\,')= \\
 -\mathcal{F}^{-1}_{\xi '\to \widetilde{y}\, '}
\Big\{\Pi'\Big[
 \widehat{\mathfrak{S}} ({{\mathbf{\widetilde N}} )}
 \big[{\widehat{\mathfrak{S}}^{^{(+)}}({\mathbf{\widetilde{B}}} )}\big]^{-1}\;\Pi^+ 
\Big(\big[\widehat{\mathfrak{S}}^{^{(-)}}({\mathbf{\widetilde{B}}} )\big]^{-1}
\widehat{\mathfrak{S}} ({\mathbf{\widetilde{P}}} )
\Big) \Big]( \xi ')
\mathcal{F}_{\widetilde{x }\, '\to \xi '}\widetilde{\psi} \Big\}.
\end{multline*}
With the help of these relations equation \eqref{o-1} can be rewritten in the following form
\begin{eqnarray}
\label{o-3}
\mathcal{F}^{-1}_{\xi '\to \widetilde{y} '}\,\big[\,\widehat{e}\,(\xi ')\,
\mathcal{F}(\widetilde{\psi} )(\xi ')\big]
=\widetilde{F} (\widetilde{y} \, ')\;\;\;\text{on}\;\;\;\mathbb{R}^2,
\end{eqnarray}
where
\begin{eqnarray}
\label{o-4}
\widehat{e}(\xi ')= e \Big((1+|\xi '|)\,\omega\Big),\;\;\;\omega=\frac{\xi '}{|\xi '|},
\end{eqnarray}
with $e$ being a homogeneous function of order $-1$ given by the equality
\begin{align}
\label{16}
 e ( \xi ' )  \!=\!
 -\Pi\,'\left\{{\mathfrak{S}} (\mathbf{\widetilde N} )
\big[{\mathfrak{S}}^{^{(+)}}(\mathbf{\widetilde{B}} )\big]^{-1}\,
\Pi^+\left(\big[{{\mathfrak{S}}^{^{(-)}}(\mathbf{\widetilde{B}} )}\big]^{-1}
{{\mathfrak{S}} (\mathbf{\widetilde{P}})}
\right)\right\}( \xi ' )
 -{\mathfrak{S}} (\mathcal{\widetilde{V}} )( \xi ' ), \quad \forall\, \xi ' \neq 0. 
\end{align}
If the function $\det\,e(\xi ')$ is different from zero for all $\xi '\neq 0$, then
$\det\,\widehat{e}(\xi ')\neq 0$ for all $\xi '\in \mathbb{R}^2$, and the corresponding
pseudodifferential operator
$$
\widehat{{\bf E}} \,:\,H^{s}(\mathbb{R} ) \to H^{s+1}(\mathbb{R} )\;\;\text{for all} \;\;s\in \mathbb{R}
$$
generated by the left hand side expression in \eqref{o-3}
is invertible. In particular, it follows that the system of
equation \eqref{14-1-dddd}-\eqref{15-dddd} is uniquely solvable with respect to $(\widetilde{u },\,\widetilde{\psi} )$  in the space
$H^{1}(\mathbb{R}^3_{+})\times H^{-1/2}(\mathbb{R}^2 )$ for arbitrary right hand sides
$(\widetilde{F}_1 ,\,\widetilde{F}_2 )\in H^1(\mathbb{R}^3_+ )\times H^{1/2}(\mathbb{R}^2 )$.
Consequently, the operator  $\widehat{\mathfrak{\widetilde{D}}}$ in \eqref{ddd-1} is invertible,
which implies that the operator \eqref{o-15} possesses a left and right regularizer.
In turn this yields that the operator \eqref{15-ddd} possesses  a left and right regularizer as well. Thus the operator \eqref{15-ddd} is Fredholm if
\begin{eqnarray*}
\det\, e ( \xi ' ) \neq 0 \qquad \forall\, \xi ' \neq 0.
\end{eqnarray*}
This condition is called the {\it \v{S}apiro-Lopatinski\u{i} condition} (cf. \cite{Esk}, Theorems 12.2 and 23.1, and also formulas (12.27),(12.25)).
Let us show that in our case the \v{S}apiro-Lopatinski\u{i} condition holds.
 To this end let us note that the principal homogeneous symbols
${\mathfrak{S}} {}(\mathbf{\widetilde N}),$  ${\mathfrak{S}} {}(\mathbf{\widetilde{B}}),$ ${\mathfrak{S}} {}(\mathbf{\widetilde{P}}),$ and
${\mathfrak{S}} {}(\mathcal{\widetilde{V}})$
of the operators  $\mathbf{N}$, $\mathbf{B}$, $\mathbf{P} $, and $\mathcal{V}$
in the chosen local co-ordinate system involved in formula \eqref{16} read as:
\begin{align*}
&
{\mathfrak{S}} {}(\mathbf{\widetilde N} )(\xi )=
 |\xi |^{-2}{\widetilde{A}  (\xi )}-\widetilde{\boldsymbol\beta} ,\quad
{\mathfrak{S}} {}(\mathbf{\widetilde{B}})(\xi )=|\xi |^{-2}{\widetilde{A}  (\xi)},\quad
{\mathfrak{S}} {}(\mathbf{\widetilde{P}} )(\xi )
=- |\xi|^{-2} \,I,
\quad
{\mathfrak{S}} {}(\mathcal{\widetilde{V}} )(\xi ')=\frac{1}{2|\xi '|}\, I,\quad
\xi =(\xi ',\xi_3 ), \;\;\xi '=(\xi_1 ,\xi_2 ),
\end{align*}
where  $\widetilde{\boldsymbol\beta}$ denotes the matrix ${\boldsymbol\beta}$ written in chosen local co-ordinate system.
 Rewrite \eqref{16} in the form
\begin{align}
e(\xi') =-\Pi'\left\{\Big({\mathfrak{S}} (\widetilde{\mathbf{B}})-\widetilde{\boldsymbol\beta}\Big)
[{\mathfrak{S}}^{^{(+)}}(\mathbf{\widetilde{B}})]^{-1}
\Pi^+\Big( [{\mathfrak{S}}^{^{(-)}}(\mathbf{\widetilde{B}})]^{-1} {\mathfrak{S}} (\mathbf{\widetilde{P}}) \Big)\right\}(\xi')
 -{\mathfrak{S}} {}(\mathcal{\widetilde{V}})(\xi') 
=e_{1}(\xi')+e_{2}(\xi')-{\mathfrak{S}} {}(\mathcal{\widetilde{V}})(\xi'), 
\label{17}
\end{align}
where
\begin{align}
\label{17b}
&
e_{1}(\xi')=-\Pi'\left\{{\mathfrak{S}} (\mathbf{\widetilde{B}})
[{\mathfrak{S}}^{^{(+)}}(\mathbf{\widetilde{B}})]^{-1}\Pi^+\Big( [{\mathfrak{S}}^{^{(-)}}(\mathbf{\widetilde{B}})]^{-1} {\mathfrak{S}} (\mathbf{\widetilde{P}}) \Big)\right\}(\xi'),\\
\label{17qq}
&
e_{2}(\xi')=\widetilde{\boldsymbol\beta}\,\Pi'\left\{[{\mathfrak{S}}^{^{(+)}}(\mathbf{\widetilde{B}})]^{-1}\Pi^+\Big( [{\mathfrak{S}}^{^{(-)}}(\mathbf{\widetilde{B}})]^{-1} {\mathfrak{S}} (\mathbf{\widetilde{P}}) \Big)\right\}(\xi')      ,\\
\label{17pp}
&
{\mathfrak{S}}{}(\mathcal{\widetilde{V}})(\xi')=\frac{1}{2|\xi'|}\,I.
\end{align}

Direct calculations give
\begin{align}
\Pi^+\Big( [{\mathfrak{S}}^{^{(-)}}(\mathbf{\widetilde{B}})]^{-1} {\mathfrak{S}} (\mathbf{\widetilde{P}}) \Big)(\xi')
&=\frac{i}{2\pi}\;\lim\limits_{t\to 0+}\;
\int_{-\infty}^{+\infty}\frac{\Big([{\mathfrak{S}}^{^{(-)}}(\mathbf{\widetilde{B}})]^{-1} {\mathfrak{S}} (\mathbf{\widetilde{P}})\Big)(\xi',\eta_3)\,d\eta_3}{\xi_3+i\,t-\eta_3}
\nonumber \\
&=-\frac{i}{2\pi}\;\lim\limits_{t\to 0+}\;
\int_{-\infty}^{+\infty}\frac{[{\mathfrak{S}}^{^{(-)}}(\mathbf{\widetilde{B}})]^{-1} (\xi',\eta_3)\,d\eta_3}{(\xi_3+i\,t-\eta_3)\,(|\xi'|^2+\eta_3^2)}
=
\frac{i}{2\pi}\;\lim\limits_{t\to 0+}\;
\int_{\Gamma^-}\frac{[{\mathfrak{S}}^{^{(-)}}(\mathbf{\widetilde{B}})]^{-1} (\xi',\tau)\,d\tau}{(\xi_3+i\,t-\tau)\,(|\xi'|^2+\tau^2)}\qquad\nonumber \\
&=\frac{i}{2\pi}\;\lim\limits_{t\to 0+}\;
\frac{2\pi i\,[{\mathfrak{S}}^{^{(-)}}(\mathbf{\widetilde{B}})]^{-1} (\xi',-i|\xi'|)}{(\xi_3+i\,t+i|\xi'|)\,2\,(-i|\xi'|)}
=
-\frac{ i\,[{\mathfrak{S}}^{^{(-)}}(\mathbf{\widetilde{B}})]^{-1} (\xi',-i|\xi'|)}{\,2\, |\xi'|\,\Theta^{^{(+)}}(\xi^ {'},\xi_3)}.
\label{2cc}
\end{align}
Now from \eqref{17b}  with the help of \eqref{2cc} we derive
\begin{align}
e_{1}(\xi')&=-\Pi'\left\{{\mathfrak{S}}^{^{(-)}}{(\mathbf{\widetilde{B}})}{\mathfrak{S}}^{^{(+)}}{(\mathbf{\widetilde{B}})}
[{\mathfrak{S}}^{^{(+)}}(\mathbf{\widetilde{B}})]^{-1}
\Pi^+\Big( [{\mathfrak{S}}^{^{(-)}}(\mathbf{\widetilde{B}})]^{-1} {\mathfrak{S}} (\mathbf{\widetilde{P}}) \Big)\right\}(\xi') \nonumber\\
&=-\Pi'\left\{{\mathfrak{S}}^{^{(-)}}{(\mathbf{\widetilde{B}})}\Pi^+\Big( [{\mathfrak{S}}^{^{(-)}}(\mathbf{\widetilde{B}})]^{-1} {\mathfrak{S}} (\mathbf{\widetilde{P}}) \Big)\right\}(\xi')
=\Pi'\left\{\frac{{\mathfrak{S}}^{^{(-)}}{(\mathbf{\widetilde{B}})}}{\Theta^{^{(+)}}}\right\}(\xi')\,\,\Big(\frac{ i\,[{\mathfrak{S}}^{^{(-)}}(\mathbf{\widetilde{B}})]^{-1} (\xi',-i|\xi'|)}{\,2\, |\xi'|} \Big)  \nonumber\\
&=-\frac{1}{2\;\pi}\int_ {\Gamma^-}\frac{{\mathfrak{S}}^{^{(-)}}{(\mathbf{\widetilde{B}})(\xi',\tau)}}{\tau+i|\,\xi'|}\;d\tau\,\,
\Big(\frac{ i\,[{\mathfrak{S}}^{^{(-)}}(\mathbf{\widetilde{B}})]^{-1} (\xi',-i|\xi'|)}{\,2\, |\xi'|} \Big) \nonumber \\
&
=-i\;{\mathfrak{S}}^{^{(-)}}{(\mathbf{\widetilde{B}})(\xi',-i\,|\xi'|)}\;\frac{ i\,[{\mathfrak{S}}^{^{(-)}}(\mathbf{\widetilde{B}})]^{-1} (\xi',-i|\xi'|)}{\,2\, |\xi'|}=\frac{1}{2\,|\xi'|}\,I.
\label{2ccc}
\end{align}
Quite similarly, from \eqref{17qq}  with the help of \eqref{2cc}   we get
\begin{align*} 
e_{2}(\xi')&=\widetilde{\boldsymbol\beta}\;\Pi'\Big\{[{\mathfrak{S}}^{^{(+)}}(\mathbf{\widetilde{B}})]^{-1}\Pi^+\Big( [{\mathfrak{S}}^{^{(-)}}(\mathbf{\widetilde{B}})]^{-1} {\mathfrak{S}}(\mathbf{\widetilde{P}}) \Big)\Big\}(\xi')       
=-\widetilde{\boldsymbol\beta}\;\Pi'\Big\{\frac{[{\mathfrak{S}}^{^{(+)}}{(\mathbf{\widetilde{B}})]^{-1}}}{\Theta^{^{(+)}}}\Big\}(\xi')\,\,\Big(\frac{ i\,[{\mathfrak{S}}^{^{(-)}}(\mathbf{\widetilde{B}})]^{-1} (\xi',-i|\xi'|)}{\,2\, |\xi'|} \Big)\quad  \nonumber\\
&
=-\frac{i\,\widetilde{\boldsymbol\beta}}{2\,|\xi'|}\,\Big(-\frac{1}{2\;\pi}\int_ {\Gamma^-}\frac{[{\mathfrak{S}}^{^{(+)}}{(\mathbf{\widetilde{B}})]^{-1}(\xi',\tau)}}{\tau+i|\,\xi'|}\;d\tau \Big)
[{\mathfrak{S}}^{^{(-)}}(\mathbf{\widetilde{B}})]^{-1} (\xi',-i|\xi'|)\nonumber\\
&
=\frac{i\,\widetilde{\boldsymbol\beta}}{4\,\pi\, |\xi'|}\,\int_ {\Gamma^-}[\widetilde{A}^{^{(+)}}(\xi',\tau)]^{-1}d\tau\,\,(-2\,i\,|\xi'|)\,\,[\widetilde{A}^{^{(-)}}(\xi',-i\,|\xi'|)]^{-1}
=i\,\widetilde{\boldsymbol\beta} \,\Big\{\frac{1}{2\,\pi\,i}\,\int_ {\Gamma^-}[\widetilde{A}^{^{(+)}}(\xi',\tau)]^{-1}d\tau\Big\}\,\,[\widetilde{A}^{^{(-)}}(\xi',-i\,|\xi'|)]^{-1}.\quad\nonumber
\end{align*}
Therefore due to \eqref{17}, \eqref{17pp}, \eqref{2ccc} and Lemma \ref{L3.2} we have
\begin{eqnarray}
\label{17qqq}
e_{2}(\xi')=\frac{i}{a^{^{(+)}}(\xi')}\widetilde{\boldsymbol\beta} \,{C^{^{(+)}}(\xi')}\,\,[\widetilde{A}^{^{(-)}}(\xi',-i\,|\xi'|)]^{-1},
\end{eqnarray}
where
  $\det \widetilde{\boldsymbol\beta}\neq0$,\,\,    $\det C^{^{(+)}}(\xi')\neq0$  and
$\det \widetilde{A}^{^{(-)}}(\xi',-i\,|\xi'|)\neq0$ for  all $\xi'\neq0.$
Then it is clear that
\begin{gather*}
\det e(\xi')=-\frac{i}{\left(a^{^{(+)}}(\xi')\right)^3}\det \widetilde{\boldsymbol\beta}\,\det C^{^{(+)}}(\xi')
\det [\widetilde{A}^{^{(-)}}(\xi',-i\,|\xi'|)]^{-1}\neq0
\end{gather*}
for  all $\xi'\neq0.$

Thus, we have obtained that for the operator $\mathfrak{D}$  the
\v{S}apiro-Lopatinski\u{i}  condition holds.
Therefore, the operator
\begin{eqnarray*}
\mathfrak{D} : H^{r+1}(\Omega)\times H^{r-1/2}(S)
\to
H^{r+1}(\Omega)\times {H}^{r+1/2}(S)
\end{eqnarray*}
is Fredholm for  $r>-\frac{1}{2}$.
\end{proof}

\begin{lem}\label{Step 3}
Let $\chi\in X^\infty$. The operator $ \mathfrak{D}$ given by \eqref{15-ddd} is Fredholm with zero index.
\end{lem}
\begin{proof}
For  $t\in [0,1]$, let us consider the operator
\begin{equation*}
\mathfrak{D}_{t}:=
\left[
\begin{array}{ccc}
r_{_{\!\Omega }}\mathbf{B}_t\mathring E & \;\;\; -r_{_{\!\Omega }}V  \\
t\,\mathbf{N}^{+} \mathring E& - \mathcal{V} &  \\
\end{array}
\right]
\end{equation*}
with  $\mathbf{B}_t=\boldsymbol\beta+t\,\mathbf{N}$   and  establish  that it is homotopic to the operator $\mathfrak{D}=\mathfrak{D}_1$.
We have to check that for the operator $\mathfrak{D}_{\,t}$ the \v{S}apiro-Lopatinski\u{i} condition is satisfied for all $t\in [0,1]$.
Indeed, in this case the \v{S}apiro-Lopatinski\u{i} condition reads as
\begin{equation*}
\det e_t(\xi')\neq 0  \;\;\;\;\; \text{for  all}\;\;\;\;\; \xi'\neq 0,
\end{equation*}
where (cf. \eqref{16})
\begin{align}
e_t(\xi')=
-\Pi'\Big\{\Big({\mathfrak{S}} (\mathbf{\widetilde{B}}_t)-
\widetilde{\boldsymbol\beta}\Big)[{\mathfrak{S}}^{^{(+)}}(\mathbf{\widetilde{B}}_t)]^{-1}
\Pi^+
\Big( [{\mathfrak{S}}^{^{(-)}}(\mathbf{\widetilde{B}}_t)]^{-1} {\mathfrak{S}} (\mathbf{\widetilde{P}}) \Big)\Big\}(\xi')
-{\mathfrak{S}} {}(\mathcal{\widetilde{V}})(\xi')=
e_{t}^{(1)}(\xi')+e_{t}^{(2)}(\xi')
-{\mathfrak{S}} (\mathcal{\widetilde{V}})(\xi')
\label{1-m}
\end{align}
with
\begin{align}
e_{t}^{(1)}(\xi')&=-\Pi'\Big\{{\mathfrak{S}} (\mathbf{\widetilde{B}}_t)
[{\mathfrak{S}}^{^{(+)}}(\mathbf{\widetilde{B}}_t)]^{-1}
\Pi^+\Big( [{\mathfrak{S}}^{^{(-)}}(\mathbf{\widetilde{B}}_t)]^{-1} {\mathfrak{S}} (\mathbf{\widetilde{P}}) \Big)\Big\}(\xi')
=\frac{1}{2\,|\xi'|}\,I,
\label{7a} \\
\label{7aa}
e_{t}^{(2)}(\xi')&=\widetilde{\boldsymbol\beta}\,\Pi'\Big\{[{\mathfrak{S}}^{^{(+)}}(\mathbf{\widetilde{B}}_t)]^{-1}\Pi^+\Big( [{\mathfrak{S}}^{^{(-)}}(\mathbf{\widetilde{B}}_t )]^{-1} {\mathfrak{S}} (\mathbf{\widetilde{P}}) \Big)\Big\}(\xi')      ,\\
\label{7aaa}
{\mathfrak{S}} {}(\mathcal{\widetilde{V}})(\xi') &=\frac{1}{2|\xi'|}\,I.
\end{align}
By direct calculations we get
\begin{align}
e_t^{(2)}(\xi')&=\widetilde{\boldsymbol\beta}\,\Pi'\Big\{[{\mathfrak{S}}^{^{(+)}}(\mathbf{\widetilde{B}}_t)]^{-1}\Pi^+\Big( [{\mathfrak{S}}^{^{(-)}}(\mathbf{\widetilde{B}}_t )]^{-1} {\mathfrak{S}} (\mathbf{\widetilde{P}}) \Big)\Big\}(\xi')\nonumber  \\
&
=-\widetilde{\boldsymbol\beta}\;\Pi'\Big\{\frac{[{\mathfrak{S}}^{^{(+)}}{(\mathbf{\widetilde{B}}_t)]^{-1}}}{\Theta^{^{(+)}}}\Big\}(\xi')\,\,\Big(\frac{ i\,[{\mathfrak{S}}^{^{(-)}}(\mathbf{\widetilde{B}}_t)]^{-1} (\xi',-i|\xi'|)}{\,2\, |\xi'|} \Big)  \nonumber\\
&
=-\frac{i\,\widetilde{\boldsymbol\beta}}{2\,|\xi'|}\,\Big(-\frac{1}{2\;\pi}\int_ {\Gamma^-}\frac{[{\mathfrak{S}}^{^{(+)}}{(\mathbf{\widetilde{B}}_t)]^{-1}(\xi',\tau)}}{\tau+i|\,\xi'|}\;d\tau \Big)
[{\mathfrak{S}}^{-} (\mathbf{\widetilde{B}}_t)]^{-1} (\xi',-i|\xi'|)\nonumber\\
&
=\frac{i\,\widetilde{\boldsymbol\beta}}{4\,\pi\, |\xi'|}\,\int_ {\Gamma^-}[\widetilde{A}^{^{(+)}}_t(\xi',\tau)]^{-1}d\tau\,\,(-2\,i\,|\xi'|)\,\,[\widetilde{A}^{^{(-)}}_t(\xi',-i\,|\xi'|)]^{-1}
\nonumber\\
&
=i\,\widetilde{\boldsymbol\beta} \,\Big\{\frac{1}{2\,\pi\,i}\,\int_ {\Gamma^-}[\widetilde{A}^{^{(+)}}_t(\xi',\tau)]^{-1}d\tau\Big\}\,\,[\widetilde{A}^{^{(-)}}_t(\xi',-i\,|\xi'|)]^{-1},
\label{d}
\end{align}
where $\widetilde{A}_t(\xi)=(1-t)\,|\xi|^2\,\widetilde{\boldsymbol\beta} + t\,\widetilde{A}(\xi)$,
$\widetilde{A}_t(\xi',\xi_3)=\widetilde{A}^{^{(-)}}_t(\xi',\xi_3)\widetilde{A}^{^{(+)}}_t(\xi',\xi_3)$ and  $\widetilde{A}^{^{(\pm)}}_t(\xi',\xi_3)$ are the "plus" and "minus"
polynomial matrix factors in $\xi_3$ of the polynomial symbol matrix $\widetilde{A}_t(\xi',\xi_3)$.
Due to \eqref{1-m}, \eqref{7a}, \eqref{7aaa}, \eqref{d} and Lemma \ref{L3.2} we have
\begin{equation*}
e_{t}^{(2)}(\xi')=\frac{i}{a_t^{^{(+)}}(\xi')}\widetilde{\boldsymbol\beta} \,{C_t^{^{(+)}}(\xi')}\,\,[\widetilde{A}^{^{(-)}}_t(\xi',-i\,|\xi'|)]^{-1},
\end{equation*}
where $ C_t^{^{(+)}}(\xi')=\big[\, c_{_{ij,t}}^{^{(+)}}(\xi')\,\big]_{ij=1}^3$ and  $c_{_{ij,t}}^{^{(+)}},$ $i,j=1,2,3,$ are  main coefficients of the   co-factors $p_{_{ij,t}}^{^{(+)}}(\xi',\tau)$ of the polynomial matrix  $\widetilde{A}^{^{(+)}}_t(\xi',\tau)$  and $a^{^{(+)}}$ the coefficient at $\tau^3$ in the determinant $\det{\widetilde{A}^{^{(+)}}_t(\xi',\tau)}$.
In addition,  $$
\det \widetilde{\boldsymbol\beta}\neq0,\quad    \det C_t^{^{(+)}}(\xi')\neq0,
\quad \det \widetilde{A}^{^{(-)}}_t(\xi',-i\,|\xi'|)\neq0
$$
for  all  $\xi'\neq0$ and $t\in [0,1]$.

Then it is clear that
\begin{gather*}
\det e_t(\xi')=-\frac{i}{\left(a^+_t(\xi') \right)^3}\det\widetilde{\boldsymbol\beta}\,
 \det C_t^{^{(+)}}(\xi')  \det \,[\widetilde{A}^{^{(-)}}_t(\xi',-i\,|\xi'|)]^{-1}\neq 0
\end{gather*}
for  all
$\xi'\neq0$ and for all $t\in [0,1]$,
which implies  that for the operator $\mathfrak{D}_t$  the  \v{S}apiro-Lopatinski\u{i} condition  is satisfied.\\
Therefore the operator
\begin{eqnarray*}
\mathfrak{D}_t\; : H^{r+1}(\Omega)\times H^{r-1/2}(S)
\to
H^{r+1}(\Omega)\times {H}^{r+1/2}(S)
\end{eqnarray*}
is Fredholm for all  $r>-\frac{1}{2}$  and  for all $t\in [0,1]$.
Consequently,
\begin{eqnarray*}
{\rm{Ind}}\,\mathfrak{D}={\rm{Ind}}\,\mathfrak{D}_{1}
={\rm{Ind}}\,\mathfrak{D}_t={\rm{Ind}}\,\mathfrak{D}_0=0.
\end{eqnarray*}
\end{proof}

\subsection*{Theorem~\ref{th2.2} Proof}
Since by Lemma~\ref{Step 3} the operator $\mathfrak{D}$ is Fredholm
with zero index, its injectivity implies the invertibility.
 Thus it remains to prove that the null space of the operator $\mathfrak{D}$ is trivial
 for $r>-\frac{1}{2}$.
Assume that  $U=(u,\psi)^\top\in H^{r+1}(\Omega)\times H^{r-1/2}(S)$ is a solution to the homogeneous equation
\begin{eqnarray}
\label{18-7ddd}
\mathfrak{D}\,U=0.
\end{eqnarray}
 The operator
\begin{eqnarray*}
\mathfrak{D}\, : \,H^{r+1}(\Omega)\times H^{r-1/2}(S)
\to
H^{r+1}(\Omega)\times {H}^{r+1/2}(S)
\end{eqnarray*}
is Fredholm with index zero for all $r>-\frac{1}{2}$.
It is well known that then there exists a left regularizer
$\mathfrak{L}$ of the operator $\mathfrak{D}$,
\begin{eqnarray}
\label{18-8dd}
\mathfrak{L} \;: H^{r+1}(\Omega)\times H^{r+1/2}(S)
\to H^{r+1}(\Omega)\times {H}^{r-1/2}(S),
\end{eqnarray}
such that
$$
\mathfrak{L}\, \mathfrak{D}=I+\mathfrak{T},
$$
where    $\mathfrak{T}$  is the operator of order  $-1$   (cf. proofs of Theorems 22.1 and 23.1 in \cite{Esk}), i.e.,
\begin{eqnarray}
\label{18-9d}
\mathfrak{T} \;: H^{r+1}(\Omega)\times H^{r-1/2}(S)
\to
H^{r+2}(\Omega)\times {H}^{r+1/2}(S).
\end{eqnarray}
Therefore, from \eqref{18-7ddd}    we have
\begin{eqnarray}
\label{18-9d-1}
\mathfrak{L}\,\mathfrak{D}\,U=U+\mathfrak{T} U=0.
\end{eqnarray}

From  \eqref{18-9d} we see that
$$
\mathfrak{T}\,U\in H^{r+2}(\Omega)\times H^{r+1/2}(S).
$$
Consequently, in view of \eqref{18-9d-1}
\begin{eqnarray}
\label{pp}
U=(u,\psi)^\top\in H^{r+2}(\Omega)\times H^{r+1/2}(S).
\end{eqnarray}
If $r\geq0$, this implies $u\in H^{1,\,0}(\Omega,A)$.  If  $-\frac{1}{2}<r<0$, we iterate the above reasoning for $U$ satisfying \eqref{pp} to obtain
\begin{eqnarray}
\label{ppp}
U=(u,\psi)^\top\in H^{r+3}(\Omega)\times H^{r+3/2}(S)
\end{eqnarray}
which again implies $u\in H^{1,\,0}(\Omega,A)$. Then  we can apply the equivalence Theorem \ref{th-D-eq} to conclude that a solution $U=(u,\psi)^\top$ to the homogeneous equation
\eqref{18-7ddd}  is trivial, i.e.,
\begin{eqnarray*}
u=0 \;\;\;\; \text{in}\;\;\;\;     \Omega,\qquad \psi=0\;\;\; \text{on}\;\;\;\;      S.
\end{eqnarray*}
Thus,   Ker\,$\mathfrak{D}=\{0\}$    in    the   class
 $H^{r+1}(\Omega)\times H^{r-1/2}(S)$   and   therefore
the operator
\begin{eqnarray*}
\mathfrak{D}\;  : H^{r+1}(\Omega)\times H^{r-1/2}(S)
\to
H^{r+1}(\Omega)\times {H}^{r+1/2}(S)
\end{eqnarray*}
is invertible for all $r>-\frac{1}{2}$.
\qed\\

For localizing function $\chi$ of finite smoothness we have the following result.
\begin{cor}
\label{c2.3}
Let a localising function $\chi\in X_+^3$.   Then the  operator
\begin{eqnarray*}
\mathfrak{D}\,:\, H^{1}(\Omega)\times H^{-1/2}(S)
\to  H^{1}(\Omega)\times {H}^{1/2}(S)
\end{eqnarray*}
is  invertible.
\end{cor}
\begin{proof} 
It can be done by word for word arguments employed in the proofs of Lemmas \ref{Step 1}--\ref{Step 3} and Theorem~\ref{th2.2},  with $r=0$ and using the mapping properties of the localized potentials for a localizing function of finite smoothness (see Appendix B). 
\end{proof}

Lemma~\ref{l2.1}, Theorem~\ref{th-D-eq} and Corollaries \ref{C1} and \ref{c2.3} imply the following assertion.
\begin{cor}
Let a localising function $\chi\in X_+^3$.   Then the  operator
\begin{eqnarray*}
\mathfrak{D}\,:\, H^{1,0}(\Omega,A)\times H^{-1/2}(S)
\to  H^{1,0}(\Omega,\Delta)\times {H}^{1/2}(S)
\end{eqnarray*}
is  invertible.
\end{cor}

\appendix
\section{Classes of localising functions}
 Here we present the classes of localizing functions used in the main text (see \cite{CMN-Loc2} for details).
 \begin{defn}
\label{dA.1}
We say $\chi \in X^k$ for integer $k\geq 0$  if $\chi
(x)=\breve{\chi}(|x|)$, $\breve{\chi} \in W_1^k(0,\infty)$ and
$\varrho\breve{\chi}(\varrho) \in L_1(0,\infty)$.
We say $\chi \in X^k_+$ for integer $k\geq 1$  if $\chi \in X^k$,
${\chi} (0)=1$ and  $\sigma_\chi(\omega)>0$ for all $\omega\in\mathbb{R}$, where
 \begin{equation}\label{schi}
    \sigma_\chi(\omega):=\left\{
    \begin{array}{l}
  \displaystyle  \frac{\hat{\chi}_s(\omega)}{\omega}>0\;\;\text{for}\;\; \omega\in\mathbb{R} \setminus\{0\},\\[3mm]
   \displaystyle \int _{0}^{\infty} \varrho \breve{\chi}\,(\varrho )  \,d\varrho\;\; \text{for}\;\;\omega=0,
    \end{array}
    \right.
     \end{equation}
and  $\hat{\chi}_s(\omega)$ denotes the sine-transform of the function $\breve{\chi}$
\begin{equation}
\label{chi-sin}
 \hat{\chi}_s(\omega):=\int\limits_{0}^{\infty}
  \breve{\chi}\,(\varrho ) \;\sin(\varrho\,\omega) \,d\varrho.
  \end{equation}
\end{defn}
Evidently, we have the following imbeddings: $X^{k_1}\subset
X^{k_2}$ and $X^{k_1}_+\subset X^{k_2}_+$ for $k_1 > k_2$.
The class $X^k_+$ is defined in terms of the sine-transform. The
following lemma from \cite{CMN-Loc2} provides an easily verifiable sufficient condition for
non-negative non-increasing functions to belong to this class.
\begin{lem}
\label{lA.2}
Let $k\ge 1$. If $\chi \in X^k$, $\breve{\chi} (0)=1$,
$\breve{\chi}(\varrho)\geq 0$ for all $\varrho\in (0,\infty)$, and
$\breve{\chi}$ is a non-increasing function on $[0,+\infty)$, then
${\chi} \in X^k_+$.
\end{lem}
The following (and other) examples for $ \chi $ are presented in \cite{CMN-Loc2},
\begin{eqnarray}
\label{chi1}
  && \chi_{_1k} (x)=
 \left\{
 \begin{array}{lll}
   \displaystyle \Big[\,
 1-\frac{|x|}{\varepsilon}\,\Big]^{k} & \text{for} & |x|< \varepsilon
 ,\\
0 & \text{for} & |x|\geq \varepsilon,
 \end{array}
 \right. \\[3mm]
 \label{3.3-5d}
 && \chi_{_2} (x)=
 \left\{
 \begin{array}{lll}
    \displaystyle \exp \Big[\,{\frac{|x|^2}{|x|^2-\varepsilon^2}}\,\Big] &
    \text{for} & |x|< \varepsilon
  ,\\
0 & \text{for} & |x|\geq \varepsilon,
 \end{array}
 \right.
 \end{eqnarray}
One can observe that $\chi_{_1k} \in X^k_+$  for $k\geq1$, while $\chi_{_2} \in
X^{\infty}_+$ due to Lemma~\ref{lA.2}.

\section{Properties of localized potentials}\label{AB}
 Here we collect some assertions describing mapping properties of the localized potentials. The proofs coincide with or are similar to the ones in \cite{CMN-Loc2} and \cite[Appendix B]{CMN-IEOT2013} (see also \cite{HW}, Chapter 8 and the references therein).

 Let us introduce the boundary operators generated by the localized layer
 potentials associated
 with the localized parametrix $P (x-y)\equiv P_\chi (x-y)$

 \begin{eqnarray}
\label{3.11} &&
\mathcal{ V} \,g(y):=-\int _{S} P (x-y)\, g(x)\,dS_x, \;\;\;y\in S, \\[3mm]
 \label{3.12} &&
\mathcal{ W} \,g(y):=-\int _{S}
\big[\,T(x, \partial_x)\,P (x-y)\,\big]^\top\,
\, g(x)\,dS_x,  \;\;\;y\in S,  \\[3mm]
 \label{3.13} &&
\mathcal{W}^{\,\prime} \,g(y):= -\int _{S}
\big[\,T(y, \partial_y)\,P (x-y)\,\big]\, g(x)\,dS_x,\;\;\;y\in S,
 \\[3mm]
\label{3.14}&&
 \displaystyle \mathcal{ L}^{\pm } g(y):=
T^{\pm}(y, \partial_y)\,W g(y) ,\;\;\;y\in S.
\end{eqnarray}
\begin{thm}
\label{tB.1}
 The following operators are continuous
\begin{align}
 \label{B-d5}
 \mathcal{ P}
&: \widetilde{H}^s(\Omega) \to H^{s+2,s}(\Omega;\Delta),\quad
 -\frac{1}{2}< s<\frac{1}{2}, \quad \chi  \in X^{1}, \\
 \label{B-d6}
&: H^s(\Omega) \to H^{s+2,s}(\Omega;\Delta),\quad
 -\frac{1}{2}<s<\frac{1}{2}, \quad \chi  \in X^{1}, \\
 \label{B-d7}
&: H^s(\Omega) \to
H^{\frac{5}{2}-\varepsilon,\frac{1}{2}-\varepsilon}(\Omega;\Delta),
 \;\; \frac{1}{2}\le s<\frac{3}{2},\;\; \forall\, \varepsilon\in(0,1),\ 
  \chi  \in X^{2}, 
\end{align}
where $\Delta$ is the Laplace operator.
\end{thm}
\begin{thm}
\label{tB.4}
The following operators are continuous
\begin{align}
\label{1-ddd}
 V    &: H^{s-\frac{3}{2}}(S) \to
 H^{s}( {\mathbb{R}^3}),\quad s<\frac{3}{2}\ ,\quad
\text{if}\ \chi  \in X^1, \quad \\
 \label{LSing}
  &: H^{s-\frac{3}{2}}(S) \to
 H^{s,s-1}(\Omega^\pm;\Delta) ,
 \quad \frac{1}{2}< s<\frac{3}{2},  \quad\text{if}\ \chi  \in X^2,
 \\
\label{2-ddd}
 W   &:  H^{s-\frac{1}{2}}(S)\to H^{s}({\Omega^\pm}) ,
 \quad s<\frac{3}{2}\ ,\quad
\text{if}\;\;\chi  \in X^2,
 \\
 \label{LDoubl}
 &:  H^{s-\frac{1}{2}}(S)\to
 H^{s,s-1}(\Omega^\pm;\Delta) ,
 \quad \frac{1}{2}< s<\frac{3}{2},  \quad\text{if}\ \chi  \in X^3.
\end{align}
\end{thm}
\begin{thm}
\label{tB.5}
If $\,\chi \in X^k$ has a compact support and  $-\frac{1}{2}\leq s\leq \frac{1}{2}$, then the following localized
operators are continuous
\begin{eqnarray}
\label{1-dd}
 V    &:& H^{s}(S) \to
 H^{s+ \frac{3}{2}}( {\Omega^\pm}) \;\;\text{for}\;\; k=2,
 \\
\label{2-dd}
 W   &:&  H^{s+1}(S)\to H^{s+\frac{3}{2}}({\Omega^\pm}) \;\;\text{for}\;\; k=3.
 \end{eqnarray}
 \end{thm}

 \begin{thm}
\label{tB.7}
Let $\psi\in H^{-\frac{1}{2}}(S)$ and
 $\varphi\in H^{\frac{1}{2}}(S)$.
Then  the following jump relations  hold on $S$:
\begin{eqnarray}
&&
 \label{3.8j}
\gamma^\pm V \psi=\mathcal{V} \psi,
 \quad \chi  \in X^1,\\
 &&
 \label{3.9j}
\gamma^\pm W\varphi = \mp\,\boldsymbol{\mu}\, \varphi +\mathcal{W} \varphi,\quad \chi \in X^{2},
\\
&&
 \label{3.10j}
T^\pm V  \psi =\pm\,\boldsymbol{\mu}\, \psi+ \mathcal{W}' \psi, \quad \chi  \in X^2,
\end{eqnarray}
where
\begin{equation}
\label{mu}
\boldsymbol{\mu}(y)=[\boldsymbol{\mu}^{pq}(y)]_{p,q=1}^3
:=\frac{1}{2}\,\big[\,a^{pq}_{kj}(y)\,n_k(y)\,n_j(y)\,\big]_{p,q=1}^3,\quad y\in S,
\end{equation}
and $\boldsymbol{\mu}(y)$ is positive definite due to \eqref{1-d}.
\end{thm}
\begin{thm}
\label{tB.6}
 Let $-\frac{1}{2} \leq s\leq \frac{1}{2}$. The following  operators
\begin{eqnarray}
\label{op1}
\mathcal{ V}   &:& H^{s }(S)\to
H^{s+1}(S), \quad  \chi  \in X^2, \\
\label{op2}
 \mathcal{ W}   &:&  H^{s+1}(S)\to
 H^{s+1}(S), \quad  \chi  \in X^3,\\
 \label{op3}
  {\mathcal{W}' }  &:&  H^{s}(S)\to
  H^{s}(S),\quad  \chi  \in X^3, \\
  \label{op4}
\mathcal{ L} ^{\pm }      &:&
H^{s+1}(S)\to
H^{s}(S),\quad \chi  \in X^3,
\end{eqnarray}
are continuous.
\end{thm}

\ack
This research was supported by the grants EP/H020497/1: ''Mathematical Analysis of Localized Boundary-Domain Integral Equations for Variable-Coefficient Boundary Value Problems'' and EP/M013545/1: "Mathematical Analysis of Boundary-Domain Integral Equations for Nonlinear PDEs", from the EPSRC, UK, and by the grant of the Shota Rustaveli National Science Foundation FR/286/5-101/13, 2014- 2017: ''Investigation of dynamical mathematical models of elastic multi-component structures with regard to fully coupled thermo-mechanical and electro-magnetic fields''.


\end{document}